\documentclass[a4paper,10pt,twoside]{article}
\usepackage{array}
\usepackage[utf8]{inputenc}
\usepackage{amsmath,amsfonts,amsthm} 
\usepackage{booktabs}

\usepackage{algorithm}
\usepackage{framed}

\usepackage{algpseudocode}
\usepackage{algorithmicx}
\usepackage{diagbox}
\usepackage{marginnote}
\usepackage{tikz}
\usetikzlibrary{decorations.pathmorphing,patterns}
\usepackage{subcaption}
\usepackage{todonotes}
\usepackage{cleveref}
\usepackage{float}
\usepackage{placeins}
\usepackage{flafter}
\usepackage{colortbl} 
\usepackage{xcolor} 
\usepackage{xfrac}
\usepackage{ifpdf}
\ifpdf
\usepackage{epstopdf}
\fi
\usepackage{graphicx}

\newcommand{\N}{\mathbb{N}}
\newcommand{\R}{\mathbb{R}}

\newcommand{\Proj}{\mathcal{P}}

\DeclareMathOperator{\range}{ran}
\DeclareMathOperator{\krn}{ker}

\newcommand{\argmin}{\mathop\mathrm{arg\,min}}

\newcommand{\dup}[2]{#1 #2}                      % as a product, i.e. f x
\newtheorem{theorem}{Theorem}[section]
\newtheorem{lemma}[theorem]{Lemma}           %%Delete [thm] to re-start numbering
\newtheorem{proposition}[theorem]{Proposition} %%Delete [thm] to re-start numbering

\newtheorem{assumption}[theorem]{Assumption}
\newtheorem{example}[theorem]{Example}
\newtheorem{remark}[theorem]{Remark}
\newcommand{\cf}{cf.\ }
\title{A primal dual projection algorithm for efficient constraint preconditioning}
\author{Anton Schiela, Matthias St\"ocklein, Martin Weiser}

\begin{document}
 \maketitle

\begin{abstract}
We consider a linear iterative solver for large scale linearly constrained quadratic minimization problems that arise, for example, in optimization with PDEs. By a primal-dual projection (PDP) iteration, which can be interpreted and analysed as a gradient method on a quotient space, the given problem can be solved by computing sulutions for a sequence of constrained surrogate problems, projections onto the feasible subspaces, and Lagrange multiplier updates. As a major application we consider a class of optimization problems with PDEs, where PDP can be applied together with a projected cg method using a block triangular constraint preconditioner. Numerical experiments show reliable and competitive performance for an optimal control problem in elasticity. 

\vspace{2ex}

    
\end{abstract}

	\section{Introduction}
	
	A large class of partial differential equations and PDE constrained optimization problems can be written as constrained optimization problems in function space of the form
	\[
	\min_{x\in X} J(x) \mbox{ s.t. } x \in \mathcal V,
	\]
	where $J : X \to \R$ is some functional and $\mathcal V \subset X$ a submanifold of the function space $X$ (with dual space $X^*$), given, e.g., by an equality constraint. For example, in structural or fluid mechanics $\mathcal V$ may describe a linear or non-linear incompressibility constraint or a coupling condition. Similarly, mixed formulations can often we written as variational problems subject to constraints. In PDE constrained optimization, $\mathcal V$ may be given by the pairs $(y,u)$ of states and controls that satisfy a given state equation. 
	
	Solving this problem by an SQP-type method involves linearization, i.e., local approximation of $J$ by a quadratic function $q:X\to \R$ and a local approximation of $\mathcal V$ by an affine subspace $x_0+V=x_0+T_{x_0}\mathcal V$. Also the treatment of inequality constraints, as present, e.g., in contact problems, by an active set method finally yields a minimization problem on an affine subspace. 
	
    In all these cases, we have to solve quadratic minimization problems of the form: 
	\begin{equation}\label{eq:abstract-problem}
	\min q(x) \mbox{ s.t. } x \in x_0+V \subset X,
	\end{equation}
	where $q$ is a quadratic function on a Hilbert space $(X,\langle\cdot,\cdot\rangle)$, $V$ is a closed linear subspace of $X$. We are particularly interested in the case, where $X$ is an infinite dimensional function space or a high dimensional finite element space, so that~\eqref{eq:abstract-problem} is large scale and iterative solvers should be employed. 	For the sake of brevity we use the same notation for the continuous and discretized spaces and operators. 
 
	If $V+x_0$ is given by a linear equation of the form $Cx+c$ then \eqref{eq:abstract-problem} can be written as a linear system of equations, called a saddle-point problem, which is of the form:
	\begin{equation}\label{eq:saddle-point}
	    \begin{pmatrix}
	                M & C^*\\
	                C & 0
	    \end{pmatrix}
	    \begin{pmatrix}
	                x\\
	                p
	    \end{pmatrix}+
	    \begin{pmatrix}
	                q'(0)\\
	                c
	    \end{pmatrix}=0.
	\end{equation}
	with $M=q''(0):X\to X^*$ symmetric. If $M$ is elliptic on $V=\ker C$, and $C:X\to P^*$ surjective onto a reflexive space $P$ with adjoint $C^*: P\to X^*$ then \eqref{eq:abstract-problem} and \eqref{eq:saddle-point} are uniquely solvable, and the solutions coincide. 
	In the last decades many approaches for solving and preconditioning these systems have been developed. 	For an extensive overview of solution algorithms for saddle point problems, we refer to \cite{benzi_golub_liesen_2005}. 
	Among the most well known approaches are reduction to the Schur complement (Uzawa method, \cite{Uzawa1958}), projection to the nullspace of $C$ (projected conjugate gradients) and Krylov-space methods for symmetric indefinite problems (MINRES, \cite{paige1975solution}). Depending on the type of application and the structure of the building blocks, each of these approaches has their specific field of application and  can be equipped with specialized preconditioners. 
	
	If $q$ is elliptic on $V$, then, at least in theory, a preconditioned conjugate gradient method, \cf \cite{HestenesStiefel:1952}, which is well defined on Hilbert spaces, can be used to solve~\eqref{eq:abstract-problem}. An implementation may use a basis of $V$, if available, and solve a reduced problem. In large-scale settings, however, $V$ is often very high-dimensional, in which case this approach is not feasible. 
	Alternatively, if a linear projection $\Proj: X\to V$ (also called a \emph{constraint preconditioner}) can be implemented, we can use projected search directions. This yields the projected preconditoned cg-method (ppcg) \cite{gould2001solution}.
	
	The distinguishing feature of ppcg among other solvers, like MINRES, is that it preserves the structure of a constrained optimization problem and can detect non-convexities of the subproblem.  
	However, as a draw-back, the projection onto $V$ has to be applied exactly, or to be more accurate, $\Proj $ has to remain constant during the iteration, and $q$ is minimized on $\tilde V := \mathrm{ran}\, \Proj $. In general, a cg-method does not allow inexact application of preconditioners, which would perturb orthogonality of search directions and reduce the robustness of cg significantly. This restriction can render constraint preconditioners computationally too expensive to be used inside a cg method. 	For example, in PDE constrained optimization, each application of a constraint preconditioner of the form ~\ref{eq:exact-constraint-preconditioner}, see below, requires the accurate solution of two linear (discretized) PDEs.
	
	One remedy has been proposed in~\cite{Ridzal:2014}, where a ppcg method with inexact preconditioning was used inside an SQP method. Here robustness was retained by a long term orthogonalization scheme inside cg. In comparison, MINRES requires only a preconditioner for these PDEs. Thus, in the last years a lot of research has been conducted towards the construction of efficient preconditioners for PDE constrained optimization together with MINRES~\cite{doi:10.1137/15M1018502,e729039adfe3490bb208102c48830119,rees2010preconditioning,doi:10.1137/16M1093021}. 
	
	In this work we consider iterative solvers for the solution of problem \eqref{eq:abstract-problem}, which are based on ppcg but overcome its difficulty, decribed above. Our main idea is to replace $V$ by an easier to handle subspace $\tilde V$,
	solve \eqref{eq:abstract-problem} with $V$ replaced by $\tilde V$, and employ back-projections onto $V$. In addition, a Lagrange multiplier update is employed in the style of an augmented Lagrangian method. In the context of PDE constrained optimization, which is our main focus of application, we will see that this allows us to replace the exact PDE solves, needed by ppcg, by application of much cheaper preconditioners.

		\section{Projected pcg for optimal control of linear PDEs}
    As a motivating example and major application we consider the iterative solution of an optimal control problem, subject to a linear PDE in abstract form:
	\begin{equation}
	\label{linearQuadraticOptimalControlProblem}
	\begin{gathered}
	\min_{(y,u) \in Y \times U} J(y,u) := \frac{1}{2} \| y -y_{d} \|_{G}^2 + \frac{\nu}{2} \Vert u-u_d \Vert^2_{U} \\
	s.t. \,\, Ay -Bu =0 \in P^*.
	\end{gathered}
	\end{equation}
	Here, $(Y,\|\cdot\|_Y)$ is a Banach space, $\|\cdot\|_G$ is a semi-norm on $Y$, induced by a positive semi-definite symmetric bilinear form $\langle \cdot,\cdot\rangle_G$, such that $\|v\|_G \le c\|v\|_Y$ for all $v\in Y$. Typically, $Y$ is some Sobolev space and $\|\cdot\|_G$ is an $L_2$-norm. Further,  $(U,\langle\cdot,\cdot\rangle_U)$ is a Hilbert space. 
	
	$P$ is assumed to be a reflexive space and identified with its bidual. The linear surjective operator $A: Y \rightarrow P^*$, which models a differential operator, is assumed to be continuously invertible and $B: U \rightarrow P^*$ is linear and bounded.

	For example, in the case of linear elliptic problems we may choose $P=Y=H^1_0(\Omega)$ with $\|y\|_G=\|y\|_{L_2(\Omega)}$, $U=L_2(\Omega)$, $B:L_2(\Omega) \to H^1_0(\Omega)^*$ the adjoint of the Sobolev embedding. 
	
	Introducing the product space $X=Y\times U$, $x=(y,u)$ and the constraint operator $C:= (A, -B)$, so that $Cx= Ay -Bu$, we can
	define the  operators
	\begin{align*}
	    M_y : Y &\to Y^*, &  \qquad M_u : U &\to U^*\\
	    \dup{(M_y y)}{v} &:= \langle y,v\rangle_G & \dup{(M_u u)}{w} &:= \langle u,w\rangle_U
	\end{align*}
    and rewrite~\eqref{linearQuadraticOptimalControlProblem} equivalently as 
	\begin{align}\label{eq:quadmin}
	 \min_{x\in X} \frac12 \dup{(M x)}{x} -\dup{s_x}{x}  \quad \mbox{ s.t. } \quad Cx=0
	\end{align}
	with $x=(x_y,x_u)$, $M: X\to X^*$, and $s_x \in X^*$ defined by
    \begin{align*}
     \dup{(M x)}{x} &= \dup{(M_y x_y)}{x_y}+\dup{(M_u x_u)}{x_u}=\langle x_y,x_y\rangle_G+\nu \langle x_u,x_u\rangle_U,\\
     \dup{s_x}{x} &= \dup{s_y}{ x_y} +\dup{s_u}{x_u} = \langle y_d,x_y\rangle_G+\nu \langle u_d,x_u\rangle_U.
   \end{align*}  
Using a block notation, we may also write
\[     
     \quad 	M = \begin{pmatrix}
	M_y & 0 \\
	0 & M_u 
	\end{pmatrix},
 \quad s_x=\begin{pmatrix} s_y \\ s_u \end{pmatrix}. 
    \]
\begin{remark}    
For a general $n$-dimensional vector space $V$ with dual $V^*$ the dual pairing $\dup{l}{v}$ is implemented via the formula $\dup{l}{v} = \underline l^T\underline v$, where  $\underline l, \underline v \in \R^n$ are the coordinate representations of $l$ and $v$, respectively, for a given basis of $V$ and its dual basis of $V^*$. For such bases, the coordinate representations $\underline L$ and $\underline{L^*}$ of a linear operator $L : V\to W$ and its Banach space adjoint $L^* : W^*\to V^*$ are related by $\underline{L^*}={\underline L}\,^T$. 
\end{remark}	
	By invertibility of $A$ we may also define the bounded control-to-state-mapping $S := A^{-1}B : U \to Y$, which is very often even a compact operator.
	If  $\nu$ is positive, the existence of an optimal solution to \eqref{linearQuadraticOptimalControlProblem} follows via standard arguments.
	%Using the Lagrange formalism
	%\begin{equation*}
	%	\lagrange(y,u,p) = J(y,u) + c(y,u)p,
	%\end{equation*}
  An optimal solution solves, together with a Lagrangian multiplier $p\in P$ (called adjoint state in this context)  the following optimality system:
	\begin{align*}
	y-y_d + A^*p & = 0, \\
	u-u_d - B^*p & = 0, \\
	A  y -B  u & = 0.
	\end{align*}
	Denoting $Z=Y\times U\times P$ we have a linear system of the form $Hz=s$, where $H : Z \to Z^*$ can be shown to have a continuous inverse. 
    In the following, we will use the notation $z=(y,u,p)=(z_y,z_u,z_p)$ and $s=(s_y,s_u,s_p)$ for the single components of the product space vectors in $Z$ and $Z^*$, respectively. In addition we will denote $X=Y\times U$ and $z_x=(z_y,z_u)$ for the primal components, so that $z=(z_x,z_p)$. 
	
    With this notation, we may write our linear system in different levels of detail:
	\begin{equation}
	\label{exactSystem}
	Hz=s
		\Leftrightarrow 
		\begin{pmatrix} 
		M & C^*\\
		 C & 0
			                                             \end{pmatrix}
\begin{pmatrix}
			z_x   \\
			z_p  
			\end{pmatrix} = \begin{pmatrix}
			s_x   \\
			0 
			\end{pmatrix}
		\Leftrightarrow 
	\begin{pmatrix}
	M_y & 0 & A^* \\
	0 & M_u & -B^* \\
	A & -B & 0
	\end{pmatrix}
	\begin{pmatrix}   z_y \\ z_u \\ z_p \end{pmatrix} = \begin{pmatrix}    s_y \\ s_u \\ 0 \end{pmatrix}
\end{equation}
	For large scale problems, direct factorization of $H$ is not computationally feasible. Thus, an iterative solver has to be applied for the solution of \eqref{exactSystem}. 
	
   If $M$ is positive definite on $\ker C$, the application of a projected preconditioned conjugate gradient (ppcg) method, which implements the cg method on $\ker C$, cf. \cite{gould2001solution} is possible. This method projects the search directions in the cg method via a constraint preconditioner to $\ker C$. For suitable right hand sides and initial iterates, a ppcg method as described in \Cref{algorithmppcg} can be applied. A preconditioner of the form
   \[
    Q = 		\begin{pmatrix} 
		\tilde M & C^*\\
		 C & 0
			                                             \end{pmatrix}
   \]
   ensures that all updates and iterates remain in $\krn C$ as long as for the initial iterate $z_{0,x} \in \ker C$ holds. Be aware that neither $H$ nor $Q$ need to be positive definite on $Z$, but only on $\ker C$, 
   but it is easily verified that $\dup{(Hd_k)}{d_k}=\dup{(Md_{k,x})}{d_{k,x}}>0$ for all computed  search directions $d_k \in Z$. 
	\begin{algorithm}[H]
		\caption{Projected Preconditioned Conjugate Gradient Method (ppcg).}
		\label{algorithmppcg}
		\begin{algorithmic}[1]
			\State \textbf{Solve:} $ Hz   = s$, where $s_p=0$
			\State \textbf{Input:} initial iterate $z_0 = \begin{pmatrix}
			z_{0,x}, z_{0,p}\end{pmatrix}$ satisfying $Cz_{0,x} = 0$
			%	\State \textbf{Setting:}  initial iterate $x_0$, symmetric and positive definite matrix $A$, linear and positive definite preconditioner $Q$, and right hand site $b$.		
			\State \textbf{Initialize:} $r_0 = Hz_0 -s$, $d_0 = g_0 = -Q^{-1} r_0$. 
			\For{$k=0,1,2,\dots$}
			\State $\alpha_k \leftarrow -\frac{\dup{r_k}{g_k}}{\dup{(M d_{k,x})}{d_{k,x}} }$
			\State $z_{k+1} \leftarrow z_k +\alpha_k d_k$
			\State $r_{k+1} \leftarrow r_k + \alpha_k H d_k$
			\State $  g_{k+1} \leftarrow  -Q^{-1}r_{k+1}$
			\State $\beta_{k} \leftarrow \frac{\dup{r_{k+1}}{g_{k+1}}}{\dup{r_k}{g_k}}$
			\State $d_{k+1} \leftarrow g_{k+1} +\beta_{k} d_k$
			\EndFor
		\end{algorithmic}
	\end{algorithm}

    In view of the block structure of \eqref{exactSystem},  a straightforward constraint preconditioner is
	\begin{equation} \label{eq:exact-constraint-preconditioner}
	Q:= \begin{pmatrix}
	0 & 0& A^* \\
	0 & \tilde M_u & -B^* \\
	A & -B & 0
	\end{pmatrix},
	\qquad 
	\tilde M := \begin{pmatrix}
	0 & 0 \\
	0 & \tilde M_u 
	\end{pmatrix}.
	\end{equation}
	This ensures that all iterates in the cg method are contained in the space $\ker C$. Here $\tilde M_u$ is a preconditioner for $M_u$. 
	The application of the preconditioner $Q$ to general right hand sides $r$
	such that $g= -Q^{-1}r$ yields the decoupled system
	\begin{align}
	\label{equationBlockPreconditioner1}
	A^* g_p &= r_y, \\
	\label{equationBlockPreconditioner2}
	\tilde M_u g_u  -B^* g_p& = r_u, \\
	\label{equationBlockPreconditioner3}
	A g_y -B g_u &= r_p,
	\end{align}
	to be solved. Thus, the solution of the coupled optimality system is reduced to the solution of a sequence of partial differential equations. In view of the evaluation of $\dup{r_k}{g_k}$ needed in line 5 and 9 of Algorithm~\ref{algorithmppcg} we compute:
	\[
	 \dup{r}{g}=\dup{(Qg)}{g}=\dup{(\tilde M_ug_u)}{g_u}=\dup{(r_u-B^*g_p)}{g_u} > 0 \mbox{ if } g_u \neq 0. 
	\]
	For stability reasons, the right formula, which can be evaluated during application of $Q^{-1}$ is preferable. 
	
	The speed of convergence of ppcg is governed by the condition number of $M$ with respect to $\tilde M$ on $\ker C$.
	\begin{proposition}
	 Assume that there are constants $\Gamma_U \ge \gamma_U > 0$ with
	 \[
	  \gamma_U \dup{(\tilde M_u u)}{u}\le  \dup{(M_u u)}{u} \le \Gamma_U \dup{(\tilde M_u u)}{u}
	  \quad \forall u\in U,
	 \]
	 such that the condition number of $M_u$ with respect to $\tilde M_u$ is $\kappa_U = \Gamma_U/\gamma_U$. Then the condition number $\kappa$ of $M$ on $\ker C$ with respect to $\tilde M$  is bounded by
	 \begin{align*}
	  \kappa \le (\nu^{-1}\|S\|_{U\to G}^2+1)\kappa_U.
	 \end{align*}        
	\end{proposition}
\begin{proof}
	We compute for $x\in \mathrm{ker}\, C$, i.e. $x_y=A^{-1}Bx_u=Sx_u$:
	\begin{align*}
	  \gamma_U \dup{(\tilde M x)}{x} &\le \nu \langle x_u,x_u\rangle_U  \le \dup{(M  x)}{x}
	  =  \langle x_y,x_y\rangle_G +  \nu \langle x_u,x_u\rangle_U   \\
	  &=\langle Sx_u,Sx_u\rangle_G +  \nu \langle x_u,x_u\rangle_U
	  \le (\|S\|^2_{U\to G} +\nu)\langle x_u,x_u\rangle_U\\
	  &= (\|S\|^2_{U\to G} +\nu)\frac{1}{\nu} \dup{(M_ux_u)}{x_u}
	  \le (\nu^{-1}\|S\|^2_{U\to G} +1)\Gamma_U \dup{(\tilde M x)}{x}
	\end{align*}
	and thus:
	\[
	\kappa \le  \frac{ (\nu^{-1}\|S\|^2_{U\to G} +1)\Gamma_U}{\gamma_U}.
	\]
\end{proof}
     \begin{remark}
      Due to the influence of $\kappa_U$ on the overall condition number, it is often advocated to use a good approximation $\tilde M_u$ for $M_u$, i.e., $\kappa_U \approx 1$. Since $U$ is usually an $L_2$ space, this can be easily achieved by a Chebyshev semi-iteration applied to the mass-matrix $M_u$, preconditioned by its diagonal~\cite{wathen2008}. 
      
      Observe that difficulties arise, if $\nu$ is very small. In practice, moderate values $\nu = 10^{-1} \dots 10^{-4}$, occuring in applications, yield reasonable performance. For smaller values of $\nu$, performance degrades. 
     \end{remark}

     	We observe that the application of the preconditioner $Q$ requires solving \eqref{equationBlockPreconditioner1} and \eqref{equationBlockPreconditioner3} exactly or at least to very high accuracy to obtain iterates which are contained in the desired linear subspace $Ay-Bu=0$. Since $A$ is  a differential operator or a discretization thereof, this may become very expensive for fine discretizations. 
	
	For this reason we may try to solve \eqref{equationBlockPreconditioner1} and \eqref{equationBlockPreconditioner3} by an iterative method. However,  pcg methods can in general not be used together with variable or inexact preconditioners such as pcg itself applied to \eqref{equationBlockPreconditioner1} and \eqref{equationBlockPreconditioner3} within the application of $Q$. Otherwise,  the orthogonality properties of the search directions break down. 
	Hence, straightforward application  of iterative solvers to \eqref{equationBlockPreconditioner1} and \eqref{equationBlockPreconditioner3} is only possible, if their termination criterion is chosen very stringently. During the ppcg method, many of these solves are necessary,  which limits the efficiency of the whole iteration. This fact is well known and widely considered as a shortcoming of ppcg methods. 

    Since the close link between $H$ and its preconditioner $Q$ cannot be loosened without the risk of degradation of robustness of ppcg, we will pursue the following approach. We assume that we have a surrogate operator $\tilde A$ for $A$ available, such that \eqref{equationBlockPreconditioner1} and \eqref{equationBlockPreconditioner3} can be solved efficiently (we will discuss the elliptic case in detail below). With $\tilde C := (\tilde A, -B)$, we can define a modified operator with matching preconditioner as follows:
    \[
     	\tilde H := \begin{pmatrix} 
		M & \tilde C^*\\
		 \tilde C & 0
		                                     \end{pmatrix},
		                                     \qquad 
     	\tilde Q := \begin{pmatrix} 
		\tilde M  & \tilde C^*\\
		 \tilde C & 0
		 \end{pmatrix}.
	\]		                                             
	Then, instead of \eqref{exactSystem} we will solve the modified problem $\tilde H \tilde z = s$ or, equivalently,
		\begin{align}\label{eq:quadminMod}
	 \min_{\tilde x\in X} \frac12 \dup{(M \tilde x)}{\tilde x}-\dup{s_x}{ \tilde x} \quad \mbox{ s.t. } \quad \tilde C \tilde x=0. 
	\end{align}
    Since in general $\ker C \neq \ker \tilde C$, the solutions of \eqref{eq:quadmin} and \eqref{eq:quadminMod} do not coincide. Moreover, the solution of \eqref{eq:quadminMod} is not even a feasible point of \eqref{eq:quadmin}. We can project $\tilde x$ to $\ker C$ by computing $y_A := A^{-1}B\tilde u$ once. However, $(y_A,\tilde u)$ is only a suboptimal feasible point of \eqref{eq:quadmin}
	We thus require an iterative procedure to solve \eqref{eq:quadmin} by computing a sequence of solutions of \eqref{eq:quadminMod} and projecting back onto $\ker C$. 
	
	\section{A General Primal-Dual Projection Algorithm}

    Consider a general quadratic minimization problems of the form: 
	\begin{equation}\label{eq:problem}
	\min_x q(x) \quad\mbox{ s.t. } x \in x_0+V \subset X,
	\end{equation}
	where $q$ is a quadratic function on a Hilbert space $(X,\langle\cdot,\cdot\rangle)$, $V$ is a closed linear subspace of $X$ and $x_0+V$ is an affine subspace of $X$. In the previous section we had $V=\ker C$ and $x_0=0$.  
	
	Let $x_*$ be a solution of \eqref{eq:problem}. Computing derivatives of $q$ along admissible directions $v\in V$ yields
	\[
	0 = \dup{q'(x_*)}{v} \quad \forall v \in V,
	\]
	or in other words $q'(x_*) \in V^\circ$, where
	\begin{align*}
	V^\circ := \{ l \in X^* : l(v) = 0 \; \forall v\in V\}
	\end{align*}
	is the annihilator or polar set of $V$. This can in turn be written as follows with the help of a Lagrangian multiplier  $\lambda := -q'(x_*)$: 
	\begin{align*}
	\exists \lambda \in V^\circ : 0 &=q'(x_*)+\lambda\\
	x_* \in x_0+V
	\end{align*}
	Adding any $\lambda\in V^\circ$ does not change $q$ on $V$, so the solution of \eqref{eq:problem} remains the same if we add $\lambda$ to $q$.
	Updating this Lagrange multiplier will play a crucial role in our iterative method. 
	
	Since $q$ is quadratic, its second derivative $b:=q''(x)$ is a bilinear form that is independent of the linearization point $x_0$, and the second order Taylor expansion is exact:
	\begin{equation*}\label{eq:q-taylor}
	    q(x_0+\delta x)=q(x_0)+\dup{q'(x_0)}{\delta x}+\frac12 b(\delta x,\delta x), \qquad b=q''(x) \;\forall x\in X.
	\end{equation*}

	   As indicated in the introduction, we will replace \eqref{eq:problem} by the surrogate incremental problem
		\begin{equation}\label{eq:problem2}
		\min_{\delta y\in X} \tilde q_{x_0}(\delta y) \quad \mbox{ s.t. } \quad \delta y \in \tilde V,
		\end{equation}
		where $\tilde V$ is a closed linear subspace of $X$ that replaces $V$, and 
	\begin{equation}\label{eq:q-tilde-definition}
		\tilde q_x(\delta y) := \dup{q'(x)}{\delta y}+\frac12 \tilde b(\delta y,\delta y).
	\end{equation}
	    This is useful, if the surrogate problem can be solved more efficiently than the original one, so we may interpret $\tilde b$ as a preconditioner. 
	   In the previous section, we had $\tilde V = \ker \tilde C$.
		To guarantee existence of a solution, we assume that $\tilde b$ is elliptic on $\tilde V$. 
		
		\subsubsection*{Primal Projection}
		
		Clearly, a solution $\delta y$ of \eqref{eq:problem2} is an element of $\tilde V$, and in general infeasible for \eqref{eq:problem}. To achieve feasibility, we project $\delta y$ back to an element $\delta x\in V$ by a linear (not necessarily orthogonal) projection
		\[
		\Proj : X \to  X,
		\]
		with the spaces
		\[
		V := \range \Proj, \qquad W := \krn \Proj.
		\]
		By $\Proj^* : X^* \to X^*$ we denote its adjoint, from which we know by the closed range theorem: 
		\[
		W^\circ = \range \Proj^*, \qquad V^\circ = \krn \Proj^*.
		\]
		      As a preliminary algorithmic idea, we may iterate the following procedure: compute a minimizer $\delta y$ of \eqref{eq:problem2},
      project $\delta x = \Proj\delta y$ back to $V$, and update $x_{k+1}=x_k+\omega \delta x$ with an optimal line-search parameter $\omega$.  
	  However, this simple algorithm does not converge to a minimizer of \eqref{eq:problem}, but stagnates after some initial progress has been made: As $\dup{q'(x_k)}{v}$ becomes smaller for all $v\in V$ during this iteration, $\dup{q'(x_k)}{w}$ for $w\in W$ starts to dominate, and thus the search directions $\delta y$ point more and more into directions close to $W$. However, since $\krn \Proj=W$, most of the search direction is projected to $0$.

	\subsubsection*{Dual Projection}
    It is well known in nonlinear optimization that the introduction of a Lagrangian multiplier $\lambda\in X^*$ can eliminate such effects. We add $\lambda$ to $q$, such that (i) the solution of the original problem is not changed and (ii) adverse search directions in $W$ are removed. Property (i) is guaranteed by choosing $\lambda\in V^\circ$, since then the objective is unchanged on the admissible space $V$. Property (ii) is implied by
    \[
     \dup{(q'(x_k)+\lambda)}{w} =0 \;\; \forall w\in W,
    \]
    i.e.\ $q'(x_k)+\lambda \in W^\circ$. The choice 
    \begin{equation} \label{eq:lambda-def}
        \lambda(x_k) = -(I_X-\Proj)^* q'(x_k) 
    \end{equation}
    yields
    \[
    \dup{\lambda}{v} = - \dup{q'(x_k)}{ (I_X-\Proj)v} = 0 \quad \forall v\in V
    \]
    and
    \[
    \dup{(q'(x_k)+\lambda)}{w} =\dup{(q'(x_k) - q'(x_k) + \Proj^* q'(x_k))}{w}
    = \dup{q'(x_k)}{\Proj w}  = 0 \quad \forall w\in W,
    \]
    and therefore satisfies both requirements. This motivates the Primal Dual Projection (PDP) Algorithm~\ref{alg:PDP}.
    	\begin{algorithm}[h!]
		\caption{General PDP Algorithm.}
		\label{alg:PDP}
		\begin{algorithmic}[1]
			\State \textbf{Solve: }
			$\min q(x) \mbox{ s.t. } x \in x_0+V.$
			
			\State \textbf{Input:} initial iterate $x_0$.
			\For{$k=  1, 2, \! ...$}
		     \State $\lambda_k \leftarrow \Proj^*q'(x_k)-q'(x_k)	$ \hspace{5cm} \mbox{(dual projection)}
		     \State $\delta y_k \leftarrow \min_{v\in \tilde V} \;\tilde q_{x_k}(v)+\dup{\lambda_k}{v}$ \hspace{3.3cm} (surrogate problem)
		     \State $\delta x_k \leftarrow \Proj \delta y_k$ \hspace{6.7cm} \mbox{(primal projection)}
		     \State $	x_{k+1} \leftarrow x_k+\omega_k\delta x_k \hspace{5.8cm} \mbox{ (with $\omega_k$ defined by \eqref{eq:line-search}) }$
			\EndFor
		\end{algorithmic}
	\end{algorithm}
	
	Obviously, since $x_0\in x_0+V$ at the beginning and $\delta x_k\in V$, we see that $x_k$ remains in $x_0+V$ for the whole iteration. More precisely,
	\[
	x_k \in x_0+(\tilde V+W) \cap V
	\]
	holds, since $\Proj\delta y_k = \delta y_k-(I_X-\Proj)\delta y_k \in \tilde V +W$. Thus, to be able to reach every element of $V$ in our iteration we need the following basic condition: 
	\begin{assumption} \label{as:subspace}
	Assume that the following inclusion holds:
		\begin{equation}\label{eq:wd}
		V \subset W+\tilde V.
		\end{equation}
    \end{assumption}
		This is true, for example, if $\tilde V \supset V$ or $W+\tilde V=X$.
		
	\subsubsection*{Optimal Line Search} The well known optimal step length $\omega_k$ is chosen by solving the scalar minimization problem
	\[
	\min_{\omega \ge 0} q(x_k+\omega \delta x_k).
	\]
	Dropping the iteration index $k$ and using the linearity of $q'(x)$ in $x$, we obtain
	\begin{equation}\label{eq:line-search}
	\begin{split}
	0 & = \frac{d}{d\omega}q(x+\omega \delta x)
	    = \dup{q'(x+\omega \delta x)}{\delta x} 
	    = \dup{q'(x)}{\delta x} +\omega \dup{q'(\delta x)}{\delta x}\\
	&\Rightarrow \omega = -\frac{\dup{q'(x)}{\delta x}}{\dup{q'(\delta x)}{\delta x}}
	=-\frac{\dup{q'(x)}{\delta x}}{b(\delta x,\delta x)}
	% =-\frac{b(x,\delta x)+(l+\lambda)(\delta x)}{b(\delta x,\delta x)}=-\frac{b(x-x_*,\delta x)}{b(\delta x,\delta x)}
	\end{split}
	\end{equation}
	Thus, the next iterate $x_+=x+\omega \delta x$ satisfies 
	\begin{align}\label{eq:gradOrth}
	 0 = \dup{q'(x_+)}{\delta x} 
	 = \underbrace{\dup{q'(x_*)}{\delta x}}_{=0} + \dup{q''(x_*)(x_+-x_*)}{\delta x} = b(x_+-x_*,\delta x).
	\end{align}
    With the energy product $\langle v,w\rangle_b := b(v,w)$ we therefore obtain  $x_+-x_* \perp_b x_+-x$
    and
	\begin{equation}\label{eq:Pythagoras}
      \|x_*-x\|_b^2 = \|x_*-x_+\|_b^2+\|x_+-x\|_b^2
	\end{equation}

	\section{Convergence analysis}
	
	We will give an interpretation of Algorithm~\ref{alg:PDP} as a preconditioned gradient method. If the subspace condition Assumption~\ref{as:subspace} holds we can show
	linear convergence, characterized by a certain condition number. For simplicity we assume w.l.o.g. $x_0=0$. 
    We consider the reduced functional
	\[
	\hat f(x) := q(\Proj x) \quad \mbox{ with }  \dup{\hat f'(x)}{v} = \dup{q'(\Proj x)}{\Proj v} 
	\Rightarrow \hat f'(x)=\Proj^* q'(\Proj x)\in W^\circ.
	\]
	We see that $\hat f(x)=q(x)$ for $x\in V$. If $q$ is positive definite, the problem
	\[
	 \min_{x\in X} \hat f(x)
	\]
	has a unique solution on $V$, but of course the solution is non-unique in $X$, since
	$\Proj(v+w)=\Proj v$ and consequently $\hat f(v)=\hat f(v+w)$ for all $w \in W=\krn \Proj$. 
	
	Thus, we may interpret our problem as an optimization problem on a quotient space. Consider the equivalence class $[\xi]=\xi+W$ in the quotient space $X/W$ and define $f:X/W \to \R$ as $f([\xi]) := \hat f(\xi)$. We may write the reduced problem as
	\begin{equation}\label{eq:problemQuot}
	 \min_{[\xi]\in X/W} f([\xi]).
	\end{equation}

	Concrete implementations will, of course, always work with the representative $\Proj \xi$, but for our theoretical investigations, switching to equivalence classes is the
	clearest way to express the theoretic results. The idea is to express Algorithm~\ref{alg:PDP} as a gradient method in $X/W$ for \eqref{eq:problemQuot}.

	%\subsubsection*{Computation of a gradient step}
	The notion of gradients depends on the chosen norm. We define the norm by the scalar product
	\begin{equation}\label{eq:preconditioner}
	\langle [v],[v]\rangle := \inf_{\nu\in [v] \cap \tilde V} \tilde b(\nu,\nu),
	\end{equation}
	which can also be interpreted as a preconditioner.
	\begin{lemma} \label{lem:scalar-product-welldefined}
		If the subspace condition \eqref{eq:wd} is fulfilled, and $\tilde b$ is elliptic on $\tilde V$, then the scalar product~\eqref{eq:preconditioner} is well defined and 
		the infimum is attained:
		\[
		\forall [v]:\; \argmin\limits_{\nu\in [v] \cap \tilde V} \tilde b(\nu,\nu) \ne \emptyset
		\]
	\end{lemma}
	\begin{proof}
		Since $V\subset W+\tilde V$, for given $v\in V$ we can find $w\in W$ and $\tilde v\in \tilde V$, such that $v=-w+\tilde v$. Hence, 
		for each $v$ we have $w\in W$ and $\tilde v\in \tilde V$, such that $v+w=\tilde v$. In other words,
		$(v+W) \cap \tilde V$ is a nonempty closed subspace of $X$. Therefore, minimizing an elliptic function over $[v]\cap \tilde V$ admits a unique solution. 
	\end{proof}
	
	The gradient step direction $[\delta\xi]$ for $f([\xi])$ with respect to the scalar product~\eqref{eq:preconditioner} is given by 
	\begin{equation}\label{eq:gradient-step}
	    [\delta\xi] := \argmin_{[v]\in X/W} \frac12 \langle [v],[v]\rangle+ \dup{f'([\xi])}{[v]}.
	\end{equation} 
	
	Since $f$ is quadratic, exact linesearch can be performed. As usual, the optimal step size is
	\begin{equation}\label{eq:stepsize-SDM-XW}
	    \omega = - \frac{\dup{f'([\xi])}{[\delta\xi]}}{f''([\xi]) ([\delta\xi],[\delta\xi])}.
	\end{equation}
	Now, we can formulate the unconstrained steepest descent method given in Algorithm~\ref{alg:AbstractGM}.
	
	\begin{algorithm}
		\caption{Steepest Descent Algorithm in $X/W$.}
		\label{alg:AbstractGM}
		\begin{algorithmic}[1]
			\State \textbf{Solve:} $\min_{[\xi]\in X/W} f([\xi])$
			\State \textbf{Input:} initial iterate $[\xi_0] \in X/W$.
			\For{$k=  0, 1,  \! ...$}
			  \State $[\delta\xi_k] \leftarrow \argmin_{[v]\in X/W} 
			            \frac12 \langle [v],[v]\rangle+\dup{f'([\xi_k])}{[v]}$
			  \State
			  \State $\omega_k 
			  \quad\leftarrow- \displaystyle\frac{f'([\xi_k])[\delta\xi_k]}{f''([\xi_k]) ([\delta\xi_k],[\delta\xi_k])}$
			  \State
		      \State $[\xi_{k+1}] \leftarrow [\xi_k] + \omega_k [\delta\xi_k]$
			\EndFor
		\end{algorithmic}
	\end{algorithm}
	
	As indicated above, the primal-dual projection method and the reduced gradient method are equivalent:
    \begin{proposition}
        Let the assumptions of Lemma~\ref{lem:scalar-product-welldefined} be satisfied.
		For initial values $x_0\in V$ and $[\xi_0]\in X/W$, respectively, with $x_0 \in [\xi_0]$, the sequences of iterates $x_k$ and $[\xi_k]$ generated by Algorithms~\ref{alg:PDP} and~\ref{alg:AbstractGM}, respectively, coincide: $x_k \in [\xi_k]$.
	\end{proposition}
	\begin{proof}
	    By induction over $k$, assume that $x_k\in [\xi_k]$, which implies $x_k = \Proj \xi_k$. We will drop the iteration index $k$ for brevity. The directional derivative in~\eqref{eq:gradient-step} can be written in terms of the multiplier $\lambda$ by~\eqref{eq:lambda-def} as
	    \begin{equation}\label{eq:directional-derivative-XW}
	    \dup{f'([\xi])}{[v]} =\dup{\Proj^*q'(\Proj \xi)}{\nu} 
	    = \dup{q'(x)}{\nu} - \dup{(I-\Proj)^*q'(x)}{\nu} 
	    = \dup{(q'(x)+\lambda(x))}{\nu} 
	    \quad \forall\nu\in [v].
	    \end{equation}
	    Let $g:=q'(x)+\lambda(x)$. Under the assumptions of Lemma~\ref{lem:scalar-product-welldefined}, the steepest descent direction $[\delta\xi]$ is computed via~\eqref{eq:gradient-step} by solving
    	\begin{align*}
	    \min_{[\delta\xi]\in X/W} \frac12 \min_{\tilde v\in [\delta\xi] \cap \tilde V} \tilde b(\tilde v,\tilde v)+ \dup{f'([\xi])}{\delta\xi}
	    &= \min_{[\delta\xi]\in X/W}\min_{\tilde v\in [\delta\xi] \cap \tilde V} \left(\frac12\tilde b(\tilde v,\tilde v)+g(\tilde v)\right)\\
	    &=\min_{\tilde v\in \tilde V} \left(\frac12 \tilde b(\tilde v,\tilde v)+g(\tilde v)\right).
	    \end{align*}
	    The first equality is due to $\dup{f'([\xi])}{\delta\xi} = g(\tilde v)$ since $\tilde v\in [\delta\xi]$, and the last follows from the fact that
	    \[
	    \tilde V = X \cap \tilde V = \left(\bigcup_{[\delta\xi]\in X/W} [\delta\xi]\right) \cap \tilde V=\bigcup_{[\delta\xi]\in X/W} ([\delta\xi]\cap \tilde V),
	    \]
	    such that the minimization problems on both sides of the equality are equivalent. 
	
	    Hence, by the definition of $\tilde q_x$ in~\eqref{eq:q-tilde-definition}, we can compute a steepest descent direction for $f$ by first solving
	    \[
	    \min_{\tilde v\in \tilde V} \frac12 \tilde b(\tilde v,\tilde v)+g(\tilde v)
	    = \min_{\tilde v\in \tilde V} \dup{q'(x)}{\tilde v}+\frac12 \tilde b(\tilde v,\tilde v)
	            +\dup{\lambda(x)}{\tilde v}
	    = \min_{\tilde v\in \tilde V} \tilde q_{x}(\tilde v)+ \dup{\lambda(x)}{\tilde v}
	    \]
	    and then taking the equivalence class $[\delta \xi]=\tilde v+W$. Thus, the reduced steepest descent direction problem~\eqref{eq:gradient-step} is equivalent to \eqref{eq:problem2} defining the search direction $\delta y$ in Algorithm~\ref{alg:PDP}, which means $\delta y \in [\delta\xi]$ and consequently $\delta x = \Proj \delta y = \Proj \delta\xi \in [\delta\xi]$.

		Moreover, the step size computations~\eqref{eq:line-search} and~\eqref{eq:stepsize-SDM-XW} are equivalent: Using~\eqref{eq:directional-derivative-XW}, $\delta x\in V\subset \ker \lambda$, and $b=q''(x)$ we obtain
		\[
		- \frac{\dup{f'([\xi])}{[\delta\xi]}}{f''([\xi])([\delta\xi],[\delta\xi])}
		= - \frac{\dup{(q'(x)+\lambda(x))}{\delta x}}{q''(\Proj\xi)(\Proj\delta\xi,\Proj\delta\xi)}
		= - \frac{\dup{q'(x)}{\delta x}}{q''(x)(\delta x,\Proj\delta x)}
		= -\frac{\dup{q'(x)}{\delta x}}{b(\delta x,\delta x)}.
		\]
		Thus, $x+\omega\delta x \in [\xi]+\omega[\delta\xi]$, which completes the induction.
	\end{proof}	
	
	% (for $x\in V$ and $\delta y\in \tilde V$):
	% \begin{align*}
	%  \frac12 q''(\delta y,\delta y)&+q'(x)\delta y+\lambda(\delta y)\\
	%  &=q(x+\delta y)+\lambda(x+\delta y)-q(x).
	% \end{align*}

    It is well known that preconditioned gradient methods, equipped with the above optimal step size, converge linearly in the energy norm defined by the quadratic term $b$, as long as both the preconditioner and $b$ are positive definite~\cite{iterativeMethodsForSparseLinearSystems}. 
    In our case, the energy norm is given by
    \[
	\|[v]\|_b :=\|\Proj v\|_b=b(\Proj v,\Proj v)^{1/2}.
	\]
	Since $x_k-x_* \in V$, so that $\|x_k-x_*\|_b=\|\Proj(x_k-x_*)\|_b$, we conclude the following result:
	\begin{theorem}\label{thm:speed}
	 Assume there are constants $\Gamma \ge \gamma>0$, such that the norm equivalence  
	\begin{align}\label{eq:condition}
	\gamma \min_{w\in [v]\cap \tilde V}\tilde  b(w,w) \le b(v,v) \le \Gamma \min_{w\in [v]\cap \tilde V} \tilde b(w,w) \quad \forall v\in V
	\end{align}
	holds. Let $\kappa := \Gamma/\gamma$ denote the condition number of Problem~\eqref{eq:problem} with respect to Problem~\eqref{eq:problem2}. Then, Algorithms~\ref{alg:AbstractGM} and~\ref{alg:PDP} converge linearly in the norm induced by $b$,
	\[
	 \|x_{k+1}-x_*\|_b \le\Theta\|x_{k}-x_*\|_b,
	\]
	with a rate
	\begin{equation}\label{eq:contraction-rate}
	 \Theta = \frac{\kappa-1}{\kappa+1}. 
	\end{equation}
	\end{theorem}

	\begin{remark}
	 Alternatively,  Algorithm~\ref{alg:PDP} could be accelerated to a conjugate gradient method. In our main application below, however, $\kappa$ is moderate or small, so the effect of acceleration is also very small. 
     Compared to cg, the pure gradient method has the advantage that it allows the flexible or inexact application of preconditioners, which can be exploited by solving \eqref{eq:problem2} only up to a certain accuracy. 
    \end{remark}

    Before we apply PDP to the optimal control problem~\eqref{linearQuadraticOptimalControlProblem}, we will consider some instructive examples. 
    First, to obtain some understanding of the concept, let us discuss a very simple concrete example. 
	\begin{example}[Orthogonal projection, codimension 1]
	%Consider for illustration the case $X$, equipped with a scalar product given by $b$:
	%\[
	%\langle v,v\rangle=\|v\|^2=b(v,v)=\tilde b(v,v).
	%\]
	Consider $X$ equipped with the norm $\|x\|_b^2 := b(x,x)$ and assume $\tilde b = b$.
	Let $V=n^\perp$, $W=\mathrm{span}\, n$, and $\tilde V = \tilde n^\perp$, where $W$ and $\tilde n$ enclose the acute angle $\theta \in [0,\pi/2\mathclose [$. 
	
	Otherwise, $W\subset \tilde V$ and hence $W+\tilde V \not \supset V$. Thus, $\theta=\pi/2$ means that the subspace condition~\eqref{eq:wd} is violated. 

	{
	
	\hfill
	\includegraphics[width=0.4\textwidth]{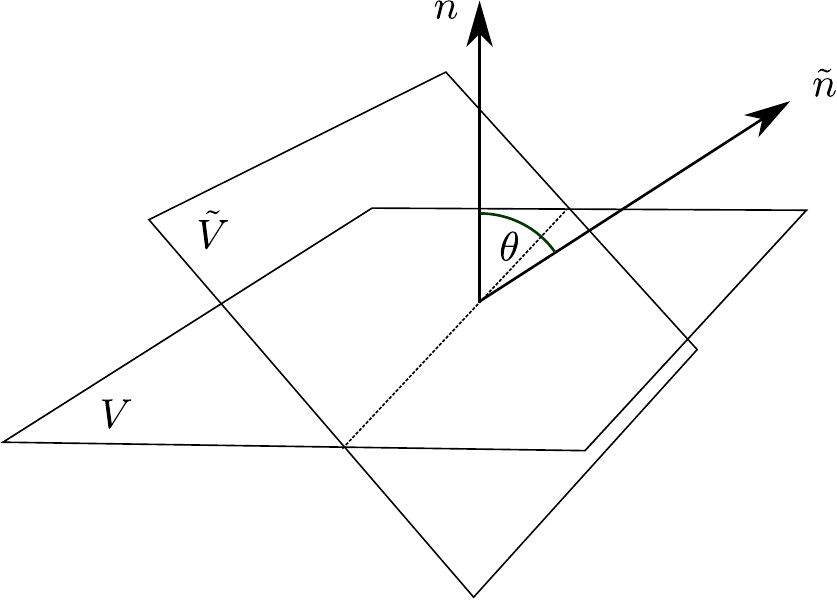}
    \hfill
	}	

If $\theta \in [0,\pi/2[$, then $[v]\cap \tilde V$ contains exactly one element, the intersection of $v+W$ and $\tilde V$, called $\tilde v=v+w$, where $w\in W$. Thus, \eqref{eq:condition} simplifies to
	\begin{align*}
	\gamma \|\tilde v\|_b^2 \le \|v\|_b^2 \le \Gamma \|\tilde v\|_b^2 \quad \forall v\in V.
	\end{align*}
Since $w\perp v$, we conclude $\|\tilde v\|_b\cos \theta \le \|v\|_b \le \|\tilde v\|_b$ and  obtain
\[
 \gamma = \cos^2\theta, \quad \Gamma=1, \quad \text{and} \;\; \kappa = \frac{1}{\cos^2 \theta}.
\]
%This holds in general, as long as $W\perp V$. 
%In our case of codimension 1 we have $0 \le \phi \le \theta$ with $\phi=0$ for $v\in V\cap \tilde V$ and $\phi=\theta$ if $v=\Proj\tilde n$, and thus for $d\ge 3$ we get
%	\[
%	\kappa = \frac1{\cos^2 \theta}.
%	\]
	Hence, the closer $V$ and $\tilde V$, the smaller $\theta$, the faster the convergence of our algorithm. If $\theta \to \pi/2$, then $\kappa \to \infty$. 
	\end{example}

	As a second example, which highlights the connection to augmented Lagrangian methods, we consider the case where constraints are dropped:
	
	\begin{example}
	Assume that $\tilde V \supset V$ and $\tilde b=b$, i.e., we drop constraints (the extreme case $\tilde V=X$ is included). 
	Let us furthermore consider the orthogonal projection $\Proj$ onto $V$ with $W=V^\perp$.
	Since $v\in V \Rightarrow v\in \tilde V$, one side of the norm equivalence~\eqref{eq:condition} is trivial:
	\[
	\inf_{\tilde v\in [v]\cap \tilde V}b(\tilde v,\tilde v) \le b(v,v) \quad \Rightarrow \quad \gamma=1.
	\]
	Also, by orthogonality, we have $\|v+w\|^2 = \|v\|^2+\|w\|^2$ for all $w\in W$
	and thus
	\[
	\inf_{\tilde v\in [v] \cap \tilde V}\|\tilde v\|^2 = \|v\|^2.
	\]
	Let us now assume that $b$ is continuous and elliptic on $\tilde V$, i.e.
	\[
	\alpha \|\tilde v\|^2 \le b(\tilde v,\tilde v) \le M \|\tilde v\|^2 \quad \forall \tilde v\in \tilde V
	\]
	holds for some $0<\alpha \le M <\infty$. For $v\in V\subset\tilde V$ we can estimate
	\[
	b(v,v) \le M \|v\|^2 = M \inf_{\tilde v\in [v]\cap \tilde V} \|\tilde v\|^2 \le \frac{M}{\alpha}\inf_{\tilde v\in [v]\cap \tilde V} b(\tilde v,\tilde v)
	\]
	and obtain $\Gamma \le M/\alpha$. We thus get the condition number bound
	\[
	\kappa \le \frac{M}{\alpha}.
	\]
	
	If $b$ is only elliptic on $V$, we may use an additional penalty term to get or to improve the required ellipticity. 
	In this case and for $\tilde V=X$ the PDP algorithm resembles an augmented Lagrangian method (cf. e.g. \cite{MR690767}). However, augmented Lagrangian methods usually rely on penalization of constraints, but dispense with a primal projection. 
	\end{example}

    \subsubsection*{Error estimate and termination}
	Let $(x_*,\lambda_*)$ be the solution to \eqref{exactSystem} and $(x_k, \lambda_k)$ the current iterate of the PDP algorithm~\ref{alg:PDP}.
	For a user-provided relative accuracy  $\Lambda_{\text{PDP}}$,
	we choose the criterion
	\begin{equation}
	\label{terminationCriterion0}
	\Vert x_* - x_k \Vert_b \leq \Lambda_{\text{PDP}} \Vert x_*-x_0 \Vert_b.
	\end{equation}
	The  solution $x_*$ is unknown, which requires suitable estimates for
	$\Vert x_*-x_0 \Vert_b$ and $\Vert x_* - x_k \Vert_b$. We will use the linear convergence of gradient methods to recall a standard error estimate, which works well in the case of fast linear convergence, i.e., contraction factors $\Theta \not\approx 1$.  
	
    In the case of linear convergence, a relation of the following form can be observed a-posteriori:
	\begin{equation}\label{eq:constantTheta}
	\Vert x_{k+1} - x_k \Vert_b = \Theta_{k} \Vert x_k -x_{k-1}\Vert_b \quad \Theta_{k} \in ]0,\Theta[.
	\end{equation}
	Of course, the upper bound $\Theta$ for the contraction as of~\eqref{eq:contraction-rate} is usually unavailable.
    In practice, the contraction factors $\Theta_k$ do not vary dramatically during the iteration, such that as a heuristic we may assume $\Theta_k=\bar\Theta$ is constant. Averaging over several steps yields the following a-posteriori estimate for the actual rate of convergence:
    \begin{equation}
     [\bar\Theta] := \left(\frac{\|x_k-x_{k-1}\|_b}{\|x_m-x_{m-1}\|_b}\right)^{1/(k-m)} \le \Theta
    \end{equation}
    The closer $\Theta_k$ is to $1$, the larger the distance $k-m$ should be taken. A reasonable preconditioner $\tilde b$ usually achieves a sufficiently small $\Theta$, such that $k-m=1$ is  appropriate. 
	
	Assuming that $\Theta_k \le [\bar\Theta]$ for all $k$, repeated application  of the orthogonality relation~\eqref{eq:Pythagoras} and the use of a geometric series implies
	\begin{align}\label{eq:gradientSum}
	 \|x_*-x_k\|_b^2 &= \sum_{i=1}^\infty \|x_{k+i}-x_{k+i-1}\|^2_b \le \sum_{i=1}^\infty [\bar\Theta]^{2i}\|x_{k}-x_{k-1}\|^2_b = \frac{[\bar\Theta]^2}{1-[\bar\Theta]^2} \|x_k-x_{k-1}\|^2_b,
	\end{align}
    which yields the computable error estimate
	\[
	 [\|x_*-x_k\|_b] := \frac{[\bar\Theta]}{\sqrt{1-[\bar\Theta]^2}} \|x_k-x_{k-1}\|_b
	\]
	in terms of the estimated contraction $[\bar\Theta]$. Using \eqref{eq:gradientSum} with $k=0$, we obtain asymptotically tight upper and lower estimates for $\|x_*-x_0\|_b^2$:
	\[
	 \sum_{i=1}^k \|x_{i}-x_{i-1}\|^2_b\le \|x_*-x_0\|_b^2 \le \sum_{i=1}^k \|x_{i}-x_{i-1}\|^2_b+\frac{\Theta^2}{1-\Theta^2} \|x_k-x_{k-1}\|^2_b.
	\]
    The left part yields the computable lower bound
    \[
     \lfloor \|x_*-x_0\|_b\rfloor := \sqrt{\sum_{i=1}^k \|x_{i}-x_{i-1}\|^2_b} \le \|x_*-x_0\|_b. 
    \]
    Finally, we obtain the computable termination criterion
	\begin{equation}
	\label{terminationCriterion}
	[\Vert x_* - x_k \Vert_b] \leq \Lambda_{\text{PDP}} \lfloor\Vert x_*-x_0 \Vert_b\rfloor.
	\end{equation}

   \section{PDP with modified ppcg for PDE constrained optimization}

   In this section we describe, how to apply \Cref{alg:PDP}, together with ppcg, to the optimal control problem~\eqref{linearQuadraticOptimalControlProblem}. 
   First, we describe a practical implementation of ppcg for the modified problem $\tilde H \tilde z=s$. Then we elaborate how to apply PDP as an outer iteration and provide condition number estimates and a simple strategy for accuracy matching. 
   
	\subsection{Application of ppcg to a modified problem}

     As indicated before, we replace the operator $H$ by a modified operator $\tilde H$, replacing $A$ by $\tilde A$ and $C$ by $\tilde C := \begin{pmatrix} \tilde A & -B\end{pmatrix}$ :
	\begin{equation}\label{eq:OCSurrogate}
	\tilde H :=	\begin{pmatrix}
	     M & \tilde C^*\\
	     \tilde C & 0
	  \end{pmatrix}
:=\begin{pmatrix}
	M_y & 0 & \tilde A^* \\
	0 &  M_u & -B^* \\
	\tilde A & -B & 0
	\end{pmatrix}
	\end{equation}
	This means that we minimize the given quadratic functional $q$ on the modified subspace $\tilde V := \ker \tilde C = \ker (\tilde A, -B)$ instead of on $V=\ker C=\ker (A, -B)$. Below, we will discuss, how $\tilde A$ can be defined implicitly as the exact inverse of the action of a linear iterative solver, applied to \eqref{equationBlockPreconditioner1} and \eqref{equationBlockPreconditioner3}. 
	
    Then, ppcg can be applied to the modified system $\tilde H z = s$. Employing a preconditioner $\tilde M_u$ for $M_u$, we use the constraint preconditioner
	\begin{equation}\label{eq:ppcg-preconditioner}
	\tilde Q:= \begin{pmatrix}
	     \tilde M & \tilde C^*\\
	     \tilde C & 0
	  \end{pmatrix}:=\begin{pmatrix}
	0 & 0& \tilde A^* \\
	0 & \tilde M_u & -B^* \\
	\tilde A & -B & 0
	\end{pmatrix},
	\end{equation}	
	which operates exactly on the modified constraint space $\tilde V$ instead of the original constraint preconditioner $Q$ from~\eqref{eq:exact-constraint-preconditioner}. There are various possibilities to choose $\tilde A$ such that $\tilde A^{-1}$ and $\tilde A^{-*}$ can be evaluated exactly and efficiently. We will discuss the case of an elliptic equation in the following section. 
	
     Next, we discuss the application of ppcg \Cref{algorithmppcg} to the modified problem  involving $\tilde H$. In many instances, $\tilde A$ is not available, but the application of its inverse $\tilde A^{-1}$ to a vector is. Hence, we have to find an indirect way to compute the application of $\tilde H$ to the search direction $d$ needed in line 7 of the ppcg \Cref{algorithmppcg}. Fortunately, this is possible by a recursive formula.  
	
	Consider the following splitting, which takes into account that $\tilde A d_y-Bd_u=0$: 
	\begin{equation*}
	\tilde{H} d = 
	\begin{pmatrix}
	M_y & 0 & \tilde{A}^*\\
	0 & M_u &-B^*\\
	\tilde{A} & -B& 0
	\end{pmatrix}
	\begin{pmatrix}
	d_y \\
	d_u \\
	d_p
	\end{pmatrix}
	=
	\begin{pmatrix}
	M_y d_y \\
	M_u d_u-B^*d_p\\
	0
	\end{pmatrix}
	+
	\begin{pmatrix}
	\tilde A ^* d_p\\
	0\\
	0 
	\end{pmatrix}.
	\end{equation*}
	Observe that the first vector on the right hand side can be evaluated, while the second vector would require the application of $\tilde A^*$, which is not available. However, application of the preconditioner (line 8 of \Cref{algorithmppcg}: $g_k=-\tilde Q^{-1} r_k$) implies
	\[
	 \tilde A^* g_{k,p}=-r_{k,y}.
	\]
    Due to $d_{k+1}=g_{k}+\beta_{k}d_k$ we obtain
    \begin{align*}
     \tilde A ^* d_{k+1,p}&=\tilde A^* g_{k,p}+\beta_{k}\tilde A^* d_{k,p}\\
     &= -r_{k,y}+\beta_{k}\tilde A^* d_{k,p}.
    \end{align*}    
    Hence, for the new variable $w_k := \tilde A^* d_{k,p}$ we obtain the recursion
    \[
     w_{k+1}=-r_{k,y}+\beta_{k}w_k.
    \]
    Since $d_0:=g_0$, the recursion is started with $w_0=-r_{0,y}$. 
    Now we can compute $\tilde H d_k$ in line~7 of \Cref{algorithmppcg} as 
	\begin{equation*}
	\tilde{H} d_k=
	\begin{pmatrix}
	M_y d_{k,y}+w_k \\
	M_u d_{k,u}-B^*d_{k,p}\\
	0
	\end{pmatrix}.
	\end{equation*}
    Finally, for a given initial guess $x_0$, such that $\tilde A y_0-Bu_0=0$, we start with $z_0=(z_{0,x},0)$, i.e. $z_{0,p}=0$, so that $\tilde H z_0$ can be evaluated without evaluation of $\tilde A$ and $\tilde A^*$. In most cases $z_0=0$ is used. 
    
	In summary, we have constructed a modified ppcg method that operates on $\ker \tilde C$.
	The resulting algorithm is presented in \Cref{algorithmModifiedppcg}.
	\begin{algorithm}[h]
		\caption{Modified ppcg without evaluation of $\tilde A$}
		\label{algorithmModifiedppcg}
		\begin{algorithmic}		
			\State \textbf{Solve:} $\tilde H z= s$
			\State \textbf{Intup} initial iterate $z_0 = \begin{pmatrix}
			x_0, 0 \end{pmatrix}$ satisfying $\tilde Cx_0 = 0$
			\State \textbf{Initialize: } $r_0 = \tilde Hz_0- s$, $d_0 = g_0 = -\tilde Q^{-1}r_0$, $w_0=-r_{0,y}$
			\For{$k=0,1,2\dots$ }
			\State $h_{k,x} \leftarrow 	M d_{k,x}+\begin{pmatrix}
	w_k \\
	-B^*d_{k,p}
	\end{pmatrix}$
			\State $\alpha_k \leftarrow -\displaystyle\frac{\dup{r_k}{g_k}}{\dup{(M d_{k,x})}{d_{k,x}}}$
			\State $z_{k+1}  \leftarrow  z_k + \alpha_k d_k$
			\State  $ r_{k+1,x} \leftarrow r_{k,x} + \alpha_k h_{k,x}$ 
			\State  $ g_{k+1} \leftarrow  -\tilde Q^{-1} r_{k+1}$ 
			\State $\beta_{k} \leftarrow \displaystyle\frac{\dup{r_{k+1}}{g_{k+1}}}{\dup{r_k}{g_k}}$
			\State $d_{k+1} \leftarrow g_{k+1} + \beta_{k} d_k$
			\State $w_{k+1} \leftarrow -r_{k+1,y}+\beta_{k}w_k$
			\EndFor1975
		\end{algorithmic}
	\end{algorithm}

	\subsection{Application of PDP}
    We are now ready to solve the constrained optimization problem~\eqref{linearQuadraticOptimalControlProblem}, as a specific form of~\eqref{eq:problem} and corresponding to $Hz=s$, with the PDP \Cref{alg:PDP} by using the auxiliary problem~\eqref{eq:problem2}, corresponding to $\tilde Hz=s$.

	Let us specify, in detail, how to apply  \Cref{alg:PDP} in the context of optimal control. This will result in Algorithm~\ref{alg:PDPOC}, below. Since the objective function $J$ is already quadratic, we have $q=J$. Also, we use $\tilde q = q$ and consequently $\tilde b = b$. Thus, the original problem~\eqref{eq:problem}  and the surrogate problem~\eqref{eq:problem2}  only differ in the choice of spaces here: $\tilde V \neq V$. 
	
	The original linear subspace $V$ is given as
	\[
	V=\{(y,u) \in X \mid Ay-Bu=0\}=\{(Su,u) \mid u\in U\}
	\]
	in terms of the solution operator or control-to-state mapping $S=A^{-1}B$.
	
		The modified subspace $\tilde V$ results from replacing $A^{-1}$ by $\tilde A^{-1}$, 
	\[
	\tilde V=\{(y,u) \in X \mid \tilde Ay-Bu=0\}=\{(\tilde S u,u) \mid u\in U\},
	\]
	with the inexact control-to-state mapping $\tilde S:= \tilde A^{-1}B$, which is realized by a preconditioner or iterative linear solver. 
	
	We define the required (non-orthogonal) projection
	\begin{equation}
	\label{projection}
	 \Proj x=\begin{pmatrix} Su  \\ u \end{pmatrix}.
	 \end{equation}
	This can be written in block operator form as
	\[
	\Proj  = \begin{pmatrix}
	0 & S\\
	0 & I_U
	\end{pmatrix} \qquad \Rightarrow \qquad (I_X-\Proj )^* =\begin{pmatrix}
	I_{Y^*} & 0\\
	-S^* & 0
	\end{pmatrix}.
	\]
	For the nullspace of $\Proj $, we obtain 
	\[
	W = Y\times \{0\}=\{ (y,0) \in Y \times U \}. 
	\]
	Assumption~\ref{as:subspace} is satisfied due to
	\[
	W+\tilde V = \{ (y+\tilde Su,u) \mid y\in Y, u\in U\} = X.
	\]

	The adjoint $(I_X-\Proj )^*$ can be applied to $q'(x)$ by first computing the adjoint state $p$ via
\begin{equation}
\label{computationAdjoint}
	A^* p = - q'_y(x)
	\end{equation}
	and then computing the Lagrange multiplier, defined in~\eqref{eq:lambda-def} as $\lambda(x) = -(I_X-\Proj)^* q'(x)$,1975 via
	\[
	\lambda = \begin{pmatrix}
	A^* p\\
	-B^* p
	\end{pmatrix}
	= - \begin{pmatrix}
	q'_y(x)\\
	B^* p
	\end{pmatrix}.
	\]

	The solution of the system $\tilde H z=s$, where $\tilde H$ is given in \eqref{eq:OCSurrogate} realizes the solution of the surrogate problem~\eqref{eq:problem2} and thus the application of the preconditioner~\eqref{eq:preconditioner}. 
	We observe that for any $x=(y,u)\in X$, the set $[x]\cap \tilde V$ contains only the single element $\tilde x = (\tilde S u,u)\in \tilde V$.
	Hence, in view of~\eqref{eq:preconditioner} we obtain the preconditioner:
	    \[
        \langle [x],[x]\rangle=\inf_{\tilde x\in [x]\cap \tilde V} \tilde b(\tilde x,\tilde x)=b(\tilde x,\tilde x)=\dup{(M\tilde x)}{\tilde x}=\dup{(M_y \tilde S u)}{\tilde Su}+\dup{(M_uu)}{u}=\|\tilde Su\|_G^2+\|u\|^2_U,
      \]
      for $b(x,x)$ on $V$, where $x=(Su,u)\in V$:
      \[
        b(x,x)=\dup{(M x)}{x}=\dup{(M_y S u)}{Su}+\dup{(M_uu)}{u}=\|Su\|_G^2+\|u\|^2_U.
      \]
	
	%As for the corresponding bilinear forms we  have the splitting:
	%\[
	%\tilde b(x,x)=b(x,x)=(Mx)(x)=(M_yy)(y)+(M_uu)(u)
	%\]
% 	If we set $y:=Su:=A^{-1}Bu$, then we can write for $x=(Su,u) \in V$:
% 	\[
% 	q(x)=q_y(Su)+q_u(u)
% 	\]
%    The preconditioner $\tilde M$ (realizing the metric~\eqref{eq:preconditioner} in the abstract steepest descent method) is defined by
%    Since $[v]\cap \tilde V$ consists just of the single element $\tilde x = (\tilde Su,u)$, we obtain
%    \[
%        \dup{\tilde M(x)}{x} = \dup{(M_y \tilde S u)}{\tilde S u} + \dup{(M_u u)}{u}.
%    \]
 	Finally, the step length $\omega_k$ for a given update $\delta x_k$ is computed as
	\begin{equation}\label{eq:stepOC}
	 \omega_k  = - \frac{\dup{(M x_k+C^* p_k- s_x)}{\delta x_k}  }   
			{\dup{(M \delta x_k)}{\delta x_k}}
	\end{equation}

	This yields in summary \Cref{alg:PDPOC} where we introduce the residual vector
	\[
	 r_k = \begin{pmatrix}
	        r_{k,x}\\
	        r_{k,p}
	       \end{pmatrix}
	       =
	       \begin{pmatrix}
	        r_{k,y}\\
	        r_{k,u}\\
	        r_{k,p}
	       \end{pmatrix}
        = \begin{pmatrix}
           M x_k + C^* p_k-s_x\\
           C x_k
          \end{pmatrix}=Hz_k-s.
	\]
		\begin{algorithm}[h]
		\caption{PDP for optimal control.}
		\label{alg:PDPOC}
		\begin{algorithmic}[1]		
			\State \textbf{Solve:} $Hz=s\quad $ i.e. $\quad \begin{pmatrix} M & C^*\\ C & 0\end{pmatrix}  \begin{pmatrix}
			x   \\
			p  
			\end{pmatrix} = \begin{pmatrix}
			s_x   \\
			0 
			\end{pmatrix}$, $x=(y,u),\; Cx=Ay-Bu$
			%	\State \textbf{Setting:}  initial iterate $z_0$ (y_0,u_0,p_0)^*$, right hand site $b = (b_y, b_u, 0)^*$.	
			\State \textbf{Input:}  initial iterate $z_0 =\begin{pmatrix} x_0, p_0  \end{pmatrix}= \begin{pmatrix} y_0 , u_0, p_0  \end{pmatrix}$ satisfying $C x_0 = 0$.

			\State $r_0 \leftarrow Hz_0-s$
			\For{$k=0,1,2\dots$}
			\State Solve: $A^*  \delta p_A = -r_{k,y}$  \hspace{2cm} (by an adjoint PDE solver)
			\State $p_{k+1} \leftarrow p_k+\delta p_A$
			
			\State $r_{k,x} \leftarrow r_{k,x} +C^* \delta p_A$
			\State Solve: $\tilde H \delta z_H =  \begin{pmatrix}
			-r_{k,x}   \\
			0 
			\end{pmatrix}  $ \hspace{2.2cm} (by ppcg \Cref{algorithmModifiedppcg})

			\State Solve: $ A \delta y_A =  -(r_{k,p}+C\delta x_H)$  \hspace{1cm} (by a PDE solver)

			\State $\delta x \leftarrow\begin{pmatrix}\delta y_A+\delta y_H,\delta u_H \end{pmatrix}$
			\State $\omega \leftarrow \displaystyle\frac{-\dup{r_{k,x}}{\delta x}}{\dup{(M \delta x)}{\delta x}}$
			\State $x_{k+1} \leftarrow x_k + \omega\delta x $ 
			\State $r_{k+1,x} \leftarrow r_{k,x} + \omega M\delta x$
			\State $r_{k+1,p} \leftarrow r_{k,p} + \omega C\delta x$
			\EndFor
		\end{algorithmic}
	\end{algorithm}
	
    \begin{remark}
     Let us comment on a few implementation details of Algorithm~\ref{alg:PDPOC}:
     \begin{itemize}
     \item Algorithm~\ref{alg:PDPOC} is formulated, such that the systems in line 5,8, and 9 can be solved iteratively and inexactly. The updates of $r_k$ are implemented in a way, such that the algorithm corrects this inexactness during the outer iteration. In particular, $r_{k,x}$ always includes the newest information on the Lagrangian multiplier $p$. This is crucial for the robustness of the algorithm with respect to numerical instabilities and inexact computation of the projection step. It is well known that updating the Lagrange multiplier is necessary to obtain $r_k \to 0$ as $k\to \infty$ in iterative methods for constrained problems. The inclusion of $r_{k,p}$ in line 9 avoids a drift-off, if the system in line 9 is solved inexactly. 
     
     \item Sometimes it is desirable to avoid application of $A$ and $A^*$ during Algorithm~\ref{alg:PDPOC}, for example, if $A^{-1}$ and $A^{-*}$ are given exactly via a time-stepping scheme. This can be achieved by simple modifications of lines 7,9,10, and 14. If $A^{-1}$ is available exactly, we conclude that $r_{k,p}=0$ throughout the iteration, so line 14 can be skipped.  Line~9 can then be replaced by the equation $\delta y_+\leftarrow A^{-1}B\delta u_H$, and line~10 accordingly by $\delta x\leftarrow(\delta y_+,\delta u_H)$.  Line 7 can be modified as follows: we may update $r_{k,y} \leftarrow 0$, using $A^*\delta p_A+r_{k,y}=0$ due to line 5, and $r_{x,u} \leftarrow r_{x,u}-B^*\delta p_A$. 
     \end{itemize}
    \end{remark}

		\subsubsection*{Condition number estimates}
	
	To give a more concrete estimate for the condition number, we recall from~\eqref{linearQuadraticOptimalControlProblem} that  $b(x,x) = (M x)x = \|y\|_{G}^2+\nu \|u\|_{U}^2$ with $y=Su$.
We would like to obtain bounds $\Gamma\ge \gamma > 0$ for the norm equivalence~\eqref{eq:condition}, which here reads
	  \begin{equation}\label{eq:Mest}
	   \gamma \dup{(M \tilde x)}{\tilde x} \le \dup{(M x)}{x} \le \Gamma\dup{(M \tilde x)}{\tilde x} \quad \forall x=(Su,u)\in V, \mbox{ where } \tilde x=\Proj x = (\tilde Su,u)
%	   \gamma\dup{(M \tilde x)\tilde x \le (M x)x \le \Gamma(M\tilde x)\tilde x, \quad \forall x\in V:  x=\mathcal P \tilde x.
	  \end{equation}
     Then the condition number $\kappa = \Gamma/\gamma$ governs the speed of convergence of PDP by Theorem~\ref{thm:speed}.

	\begin{proposition}\label{pro:cond_optctrl}
	  The condition number $\kappa$ is bounded by
	 \begin{align*}
	    \kappa \le (1+\nu^{-1}\|S\|^2_{U\to G})(1+\nu^{-1}\|\tilde S\|^2_{U\to G}).
     \end{align*}
      If in addition 
		\begin{equation}\label{eq:relerror2}
		\|\tilde  Su- Su\|_{G} \le \varepsilon \|u\|_{U} \quad\forall u\in U
		\end{equation}
	    holds, i.e. $\|\tilde S-S\|_{U\to G} \le \varepsilon$,
		then the condition number is bounded by
	 \begin{align*}
	    \kappa \le \left(1+\frac{\varepsilon}{\sqrt{\nu}}+\frac{\varepsilon^2}{\nu}\right)^2.
     \end{align*}
	\end{proposition}
	\begin{proof}
	    We have to derive bounds for $\gamma$ and $\Gamma$ in \eqref{eq:Mest}.
	
	    Exploiting $\dup{(M \tilde x)}{\tilde  x} = \|\tilde S u\|_G^2 + \nu \|u\|_U^2 \ge \nu \|u\|_U^2$, we estimate for $u\ne 0$ on one hand
	    \[
	    \frac{\dup{(Mx)}{x}}{\dup{(M\tilde  x)}{\tilde  x}} \le \frac{\|Su\|_G^2+\nu\|u\|_U^2}{\nu\|u\|_U^2} 
	    \le \nu^{-1}\|S\|_{U\to G}^2 +1 =: \Gamma,
	    \]
	    and on the other hand,
	    \[
	        \frac{\dup{(M x)}{x}}{\dup{(M\tilde  x)}{\tilde x}} \ge \frac{\nu\|u\|_U^2}{\|\tilde S u\|_G^2 + \nu\|u\|_U^2}
	        \ge \frac{1}{\nu^{-1}\|\tilde S\|_{U\to G}^2+1} =: \gamma,
	    \]
	    obtaining $\kappa \le (1+\nu^{-1}\|S\|^2_{U\to G})(1+\nu^{-1}\|\tilde S\|^2_{U\to G})$.
	   
	    Now, let $y=Su$ and $\tilde y=\tilde Su$, and assume $\|\tilde y-y\|_G \le \varepsilon \|u\|_U$. Then we can compute, using Young's inequality $2ab \le \nu^{-1/2}a^2+\nu^{1/2}b^2$:
	    \begin{align*}
	       \|y\|_G^2 &\le (\|y-\tilde y\|_G+\|\tilde y\|_G)^2
	        = \|\tilde y\|_G^2+2\|\tilde y\|_G\|y-\tilde y\|_G+\|y-\tilde y\|_G^2\\
	       & \le \|\tilde y\|_G^2+2\|\tilde y\|_G\varepsilon \|u\|_U+\varepsilon^2\|u\|_U^2
	       \le \|\tilde y\|^2_G+\frac{\varepsilon}{\sqrt{\nu}}(\|\tilde y\|_G^2+\nu\|u\|_U^2)+\frac{\varepsilon^2}{\nu}\nu\|u\|_U^2.
	    \end{align*}
	    This yields the desired bound
	    \begin{align*}
	        \|y\|_G^2+\nu\|u\|_G^2 \le \left(1+\frac{\varepsilon}{\sqrt{\nu}}+\frac{\varepsilon^2}{\nu}\right)\left(\|\tilde y\|_G^2+\nu\|u\|_G^2 \right) =: \Gamma.
	    \end{align*}
      The same estimate can be shown with the roles of $y$ and $\tilde y$ switched, which yields the same bound for $1/\gamma$. This implies the second estimate on $\kappa$.
	\end{proof}
	\cref{pro:cond_optctrl} shows that a bounded condition  number of PDP for $\nu\to 0$ can be achieved, if the approximate solution operator $\tilde S$ is chosen such that $\varepsilon \sim \sqrt{\nu}$. 
\subsubsection*{Accuracy matching}
	
During application of the PDP Algorithm~\ref{alg:PDPOC}, the main computational steps can be performed by iterative solvers. Then, accuracy requirements have to be set for these inner solvers. In all cases, updates are computed with zero initial guess. Thus, relative error criteria of the form
\[
 \|\delta x_k-\delta x_*\| \le \Lambda \|\delta x_*\|,
\]
equipped with suitable norms, are appropriate. We need tolerances $\Lambda_{\tilde A}, \Lambda_{\text{ppcg}}, \Lambda_P$, and $\Lambda_{P^*}$ for the accuracy of $\tilde A$, the application of ppcg to the modified problem $\tilde H \delta z_H = (-r_{k,x},0)^T$, the primal projection, and the dual projection, respectively. We assume that the desired final tolerance $\Lambda_{\text{PDP}}$  is given for the relative energy error reduction as
\[
 \|x_k-x_*\|_M \le \Lambda_{\rm PDP}\|x_*-x_0\|_M. 
\]
We observe that $\Lambda_{\tilde A}, \Lambda_{\text{ppcg}}$, $\Lambda_P$, and $\Lambda_{P^*}$ mainly influence the quality of the updates and thus the number of outer iterations in PDP. Thus, these four quantities should be matched. The most straightforward idea, which we also applied in our implementation, is to choose them all equal to a common constant $\Lambda$. Then there is a trade-off between stringency of the corresponding inner computations and number of outer PDP iterations. In the numerical experiments in Section~\ref{sec:num-exp}, we observe that inner tolerances in the range $\Lambda  \sim 10^{-1}-10^{-3}$, which lead to only a few PDP iterations, are favourable concerning efficiency.

In some applications, feasibility with respect to the state equation enjoys priority, compared to optimality. In this case, $\Lambda_P$ can be chosen more stringently than the remaining constants. 

Let us discuss a discrepancy that occurs, concerning error measures. While the errors of ppcg and PDP can be measured in the same norm, namely $\|\delta x\|_M^2=\|\delta y\|_G^2+\nu\|\delta u\|_U^2$, which corresponds to the objective of the optimal control problem, the other three iterations, i.e., for the primal and dual projection and for $\tilde A$, concern the PDE under consideration. For them, natural error measures are usually taken with respect to a different norm, for example, an energy norm $\|\delta y\|_A$ in the case of an elliptic problem, which is usually stronger than $\|\delta y\|_G$. A relative error bound $\Lambda_{\tilde A}$ in the energy norm,  implies, using continuity of the identity $I_Y: (Y,\|\cdot\|_A)\to (Y,\|\cdot\|_G)$ and continuity of the solution operator $S:(U,\|\cdot\|_U)\to (Y,\|\cdot\|_A)$:
	\[
	\|y-\tilde y\|_G \le \|I_Y\|_{A,G}\|y-\tilde y\|_A \le \|I_Y\|_{A,G}\Lambda_{\tilde A} \|y\|_A \le  \Lambda_{\tilde A} \|I_Y\|_{A,G}\|S\|_{U,A}\|u\|_U.
	\]
Thus, for the solver accuracy condition~\eqref{eq:relerror2} of \Cref{pro:cond_optctrl} we obtain $\varepsilon \le c\Lambda_{\tilde A}$.

		\subsection{Comparison to MINRES based solvers}
	Recalling that the system \eqref{exactSystem} has saddle point structure, a range of alternative solvers can be applied~\cite{benzi_golub_liesen_2005}. However, as we will discuss, our method offers several advantages which makes it attractive for use in a certain class of problems.

	The MINRES method, originally developed in \cite{paige1975solution},   is one of the standard approaches to solve symmetric indefinite systems and has become a popular iterative solver for linear quadratic optimal control problems. A couple of preconditioners to be used with MINRES have been proposed~\cite{doi:10.1137/15M1018502,e729039adfe3490bb208102c48830119,rees2010preconditioning,doi:10.1137/16M1093021}. 
Two block diagonal preconditioners $Q_1$ (due to \cite{doi:10.1137/16M1093021}) and $Q_2$  (cf. \cite{rees2010preconditioning}), which we will discuss in the numerical experiments below, are given  as 
\begin{align} \label{eq:minres-preco}
    Q_1^{-1} := \begin{pmatrix}
                \tilde A^{-1} & & \\
                & \tilde M^{-1}_u & \\
                & & \tilde A^{-1}
    \end{pmatrix},\qquad 
    Q_2^{-1} := \begin{pmatrix}
                \tilde M_y^{-1} & & \\
                & \tilde M^{-1}_u & \\
                && \tilde A^{-1} M_y \tilde A^{-*}
    \end{pmatrix},
\end{align}
where again $\tilde A^{-1}$ is an approximation of $A^{-1}$ and $\tilde M_u$ is some preconditioner for  $M_u$.

We observe that $Q_1$ requires symmetry of $\tilde A$, and thus in particular $P=Y$. $Q_2$ can only be applied straighforwardly, if $\tilde M_y$ (or its discretization) is invertible, as in the case of complete observation. Also, from a functional analytic point of view, $Q_2$ shifts the spaces of $A$. For example, in the elliptic case with regular data, one has to consider $A: H^2(\Omega) \to L_2(\Omega)$ instead of $A:H^1(\Omega) \to H^1(\Omega)^*$. 

For some problem instances, in particular in the case of distributed control with complete observation on the domain, it is possible to modify $Q_2$ to a preconditioner that enjoys superior robustness with respect to small Tychonov parameters $\nu$. 
$A$ is replaced by a suitably chosen surrogate that includes the local effect of the control on the state~(cf. e.g. \cite{Zul2011,PearsonWathen2012,StollWath2012,schiela_ulbrich:2014}). 
In our context, this can be described by the choice
\[
  \tilde A :=(A+\nu^{-1/2}BE),
\]
where $E:Y \to U$ is some embedding, which exists in a natural way in some cases, e.g., if $U=L_2(\Omega), Y=H^1(\Omega)$.  

Beyond these differences, MINRES and ppcg differ in a couple of qualitative aspects. Most importantly, ppcg retains the structure of a constrained optimization problem. In contrast, MINRES, which solves a general indefinite linear system, cannot distinguish between descent and ascent directions. This difference can be important, if iterative solvers are used inside an SQP method for nonlinear optimization. In this context, inner iterative solvers are required to compute descent directions. Here, ppcg admits a couple of well known variants (e.g., truncated cg), in contrast to MINRES. 

As discussed, ppcg requires a constraint preconditioner, which defines the null space on which the problem is to be solved. This preconditioner does not need to be positive definite on the whole space, but only on the feasible subspace -- a very natural requirement, given the theoretical background of the problem under consideration. For example, our preconditioner $\tilde Q$ is only positive definite on $\ker (\tilde A, -B)$, but indefinite on $Y \times U \times P$. In contrast, preconditioners for MINRES have to be positive definite on the whole space, ignoring the subspace structure of the optimal control problem. This limits the  range of applications of MINRES with the above preconditioners.

Thus, to summarize, our new solver has its field of application, where the optimization structure of the problem plays an important role, in particular as an inner solver for nonlinear, potentially nonconvex problems. In some instances, however, MINRES admits a preconditioner with superior robustness properties in case of very small $\nu$. 
	
	\section{Application to Optimization with Elliptic PDEs}

	In the remaining parts of this paper we assume in addition that $P=Y$ and that $A:Y\to P^*$ corresponds to a symmetric, bounded and elliptic bilinear form $a:Y \times Y \to \R$, as in the case of linear elliptic PDEs. In this case, the approximate solution steps in line 5 and 9 of \Cref{alg:PDPOC} can be computed by application of a pcg-method applied to $A=A^*$. 

	\subsection{A Chebyshev Preconditioner}
	
	To construct $\tilde A^{-1}$ we employ the Chebyshev semi-iteration algorithm, which has been used in various contexts to construct polynomial preconditioners (cf. e.g. \cite{MR694525,MR801178}\cite[Sec. 12.3]{iterativeMethodsForSparseLinearSystems}). Here we can exploit the fact that approximate applications of $A^{-1}$ are performed repeatedly in Algorithm~\ref{alg:PDPOC}.
	
    For a symmetric positive definite (spd) matrix $A$, the Chebyshev iteration requires an spd preconditioner $Q_{A}$ (for example a multigrid preconditioner), and estimates for the smallest eigenvalue $\varsigma_{\text{min}}$ and the largest eigenvalue $\varsigma_{\text{max}}$ of $Q_A^{-1} A$, \cf \cite{chebyshev}. 
    It is a classical result that $(x_k-x_*)=p_k(Q_A^{-1}A)(x_0-x_*)$, where $p_k$ is a Chebyshev polynomial with $p_k(0)=1$, which is minimal on $[\varsigma_{\min},\varsigma_{\max}]$ among all polynomials of degree $k$. In particular, $p_k$ depends only on $k$, $\varsigma_{\text{max}}$, and $\varsigma_{\text{min}}$.  Thus, after fixing these three quantities, we obtain a fixed \emph{linear} mapping 
	\[
	 \tilde A^{-1} := p_k(Q_A^{-1}A),
	\]
    which is, for appropriately chosen $k$, a good approximation of $A^{-1}$. 
	
    If the spectrum of $Q_A^{-1} A$ is contained in the positive interval $[\varsigma_{\text{min}},\varsigma_{\text{max}}]$, this iteration converges with a linear rate
    \begin{equation}\label{eq:chebyest}
     \|x_k-x_*\|_A\le 2\left(\left(\frac{\sqrt{\kappa_{\rm C}}-1}{\sqrt{\kappa_{\rm C}}+1}\right)^k+\left(\frac{\sqrt{\kappa_{\rm C}}-1}{\sqrt{\kappa_{\rm C}}+1}\right)^{-k}\right)^{-1}\|x_0-x_*\|_A,\quad \kappa_{\rm C}:=\frac{\varsigma_{\text{max}}}{\varsigma_{\text{min}}}
    \end{equation}
     which is the same rate as the a-priori estimate for the pcg-method. 
     
     % This result follows via existence of a $Q_A$-orthogonal basis of eigenvectors of $A$ and the fact that $p_k$ maps all eigenvalues of $A$ to an interval $[-\theta,\theta]$, where $\theta$ is the constant given in \eqref{eq:chebyest}.

	To obtain the required accurate estimates for the eigenvalues $\varsigma_{\text{min}}$ and  $ \varsigma_{\text{max}}$ we observe that in \Cref{alg:PDPOC} a system of the form $A^* p_{k+1}=r$ has to be solved before \Cref{algorithmModifiedppcg} is applied to $\tilde H$. Since $A=A^*$ is spd in the setting considered here, this can be done by pcg applied to $A$ with the same preconditioner $Q_A$. We can use the connection of cg with the Lanczos iteration to estimate $\varsigma_{\text{min}}$ and  $ \varsigma_{\text{max}}$: The tridiagonal matrix 
	\begin{equation*}
	T_k:= \begin{pmatrix}
	\frac{1}{\alpha_0} & \frac{\sqrt{\beta_0}}{\alpha_0} & $ $ & $ $ & $  $ \\
	\frac{\sqrt{\beta_0}}{\alpha_0} & \frac{1}{\alpha_1}+ \frac{\beta_0}{\alpha_0} & \ddots & $ $ & $ $ & \\
%	$ $ & \cdot & \cdot & \cdot& $ $  \\
    $ $ & \ddots & \ddots & \frac{\sqrt{\beta_{k-1}}}{\alpha_{k-1}} \\
	 $ $ & $ $ & \frac{\sqrt{\beta_{k-1}}}{\alpha_{k-1}} & \frac{1}{\alpha_{k} } +\frac{\beta_{k-1}}{\alpha_{k-1}}
	\end{pmatrix}\in \R^{k+1\times k+1},
	\end{equation*}
	consisting of the quantities $\alpha_i$ and $\beta_i$ computed within the cg algorithm, is also produced by the Lanczos method applied to $A$ and preconditioned by $Q_A$ (cf.~\cite[Chapter 6]{iterativeMethodsForSparseLinearSystems}). The eigenvalues of $T_k$
	yield adequate estimates for the eigenvalues of $Q_A^{-1}A$.
	Assuming that $Q_A$ is a reasonable preconditioner, $Q_A^{-1}A$ is moderately conditioned, so the extreme eigenvalues of $T_k$ approximate the extreme eigenvalues of $Q_A^{-1}A$ well already for a moderate number $k$ of iterations (we observe a few tens), and can be computed with negligible effort by standard means. 
    The computed extreme eigenvalues can then be used to parametrize the Chebyshev iteration. Moreover, via \eqref{eq:chebyest} we can compute the smallest $k\in \N$ satisfying
   \[
     \|x_k-x_*\|_A \le \Lambda_{\tilde A}\|x_0-x_*\|_A
   \]
for a  desired relative accuracy $\Lambda_{\tilde A}$.

	\FloatBarrier
	\subsection{Numerical Example: Optimal Control of Linear Elasticity} \label{sec:num-exp}

	We test Algorithm~\ref{alg:PDPOC} at an optimal control problem in linear elasticity and also perform comparisons with MINRES. This class of problems is computationally challenging, since stiffness matrices are usually larger and more dense than encountered in scalar valued problems. For fine grids, direct solvers are hardly applicable and iterative methods have to be used. 
	
		\begin{figure}
			\centering
			\includegraphics[trim=0cm 0cm 0cm 0cm ,clip,scale = 0.065]{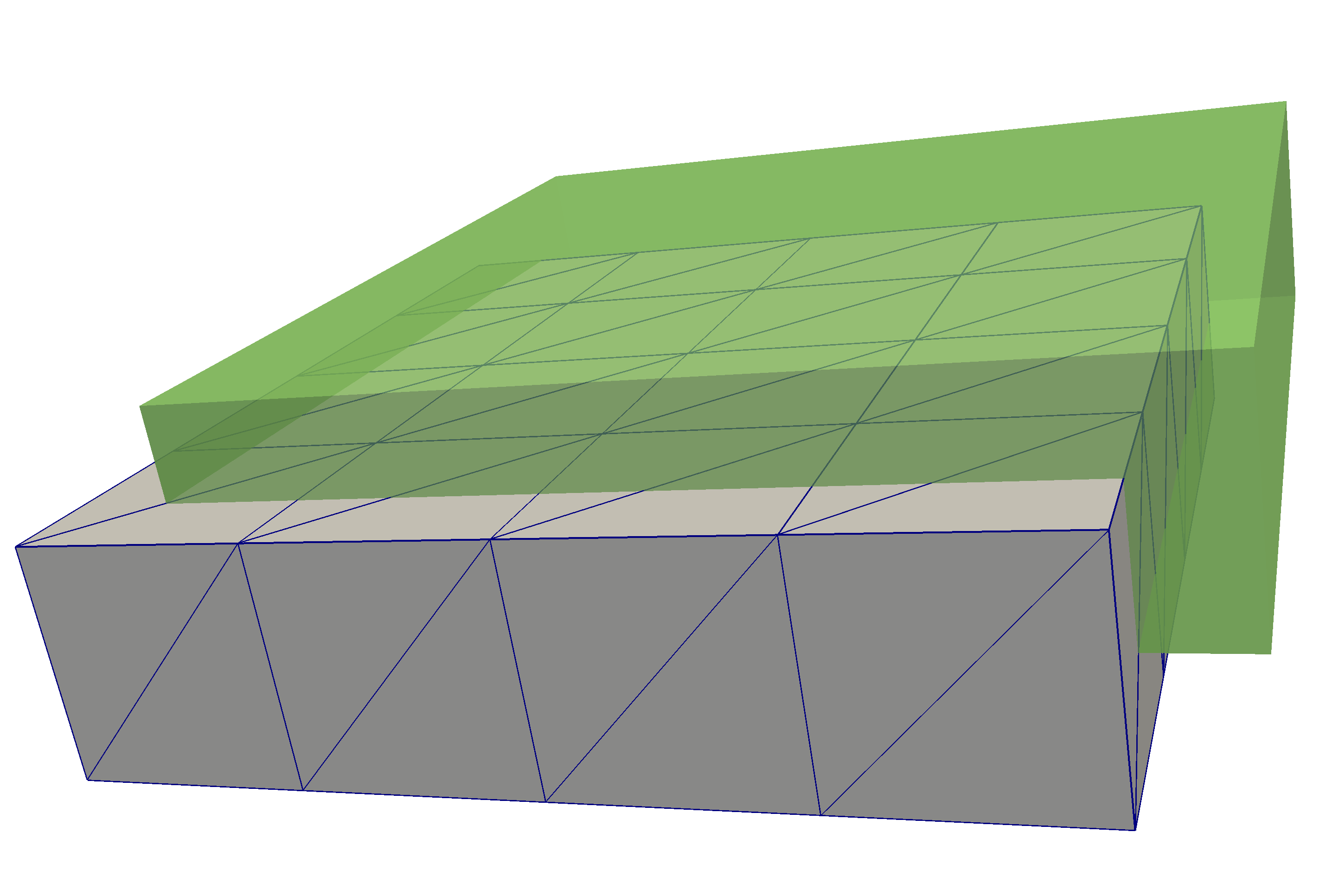} 
					
\caption{Initial coarse grid (grey) and desired deformation (green opaque). Homogeneous Dirichlet boundary conditions on the right face. Control acts as a traction force on the top face.
}
			\label{figureDesired}
		\end{figure}

	We consider the deformation of a three-dimensional body, represented by a cuboid domain $\Omega$ shown in Figure~\ref{figureDesired}.
	Its boundary $\Gamma$ consists of the two segments $\Gamma_D$ and $\Gamma_N$. $\Gamma_D$ denotes the Dirichlet boundary, where the body is clamped, and  $\Gamma_N$ represents the Neumann boundary. The control $u$ is applied as a traction force on a part of $\Gamma_N$, namely the top face. The remaining Neumann boundary conditions are homogenous. 
	For the function spaces, we choose the vector valued spaces $ Y = P = H^1(\Omega, \R^3)$ and $U = L^2(\Gamma_N,\R^3)$.
	Figure~\ref{figureDesired} also illustrates the coarse grid and the desired deformation $y_d$, a pure translation.

	The problem is discretized by linear finite elements, using various numbers of refinements of the coarse grid seen in Figure~\ref{figureDesired}, which also serves as the lowest level in our multiplicative multigrid preconditioner (MG in the following). This preconditioner is a single V-cycle with one pre- and post-smoothing step on each level. We use a $3\times 3$ block-diagonal Jacobi smoother matching the vector-valued structure of the problem. 
	This MG preconditioner yields a relative condition number $\kappa(Q_{\rm MG}^{-1}A) \approx 8.5$. 
	Alternatively, additive multigrid preconditioners such as BPX~\cite{BramblePasciakXu1990} can be used. BPX is computationally cheaper than MG, but yields larger condition numbers  $\kappa(Q_{\rm BPX}^{-1}A) \approx 100$ on fine grids.

		\begin{table}[h]\small
			\centering
			\begin{tabular} { l c c c c c c c}\toprule
			ref. & $0$ & $1$ & $2$ & $3$  & $4$  & $5$ & $6$ \\ \midrule
			dof & $471$ & $2469$ & $15\,681$ & $111\,225$  & $836\,841$  & $6\,490\,569$ & 51\,123\,081\\ 
			\midrule
							  	outer its. & 5& 5 & 6& 6 & 6 &  5 & 5\\
				 tot. MG  & 74 & 558 & 652&  727& 759 & 665 & 712\\
				 $\kappa(Q_{\rm MG}^{-1}A)$ & 1.0 & 5.2 & 7.3 & 8.2 & 8.5 & 8.5 & 8.5\\
							 \midrule 

				  	outer its. & 4& 5 & 6& 7 & 7 &  7 & 6\\
				 tot.  BPX  & 56 & 769 & 1427&  1970& 2278 & 2496 & 2750\\
				 $\kappa(Q_{\rm BPX}^{-1}A)$ & 1.0 & 17.6 & 40 & 60.7 & 79.5 & 97.4 & 115.1\\
				\bottomrule
			\end{tabular}
			\caption{Required outer iterations, preconditioner applications (V-cycle multigrid (MG) and BPX) for the PDP method for varying number of uniform refinements with $\nu=10^{-3}$, $\Lambda=10^{-2}$. In addition, estimates on condition numbers for $A$, achieved by MG and BPX are listed. 
			}
			\label{tabularIterations3}
		\end{table}
		
					\begin{table}[h]\small
			\centering
	\begin{tabular} { l c c c c c c c c c}\toprule
			$\Lambda$  & $0.3$ & $10^{-1}$  & $3\cdot 10^{-2}$ & $10^{-2}$ & $3\cdot 10^{-3}$ & $10^{-3}$ &  $3\cdot 10^{-4}$ & $10^{-4}$\\ \midrule
				  	outer its.  & 22 & 11 & 7 & 6 & 7& 4 & 3& 3\\
				  	\midrule
                 MG: $\tilde H^{-1}$   & 696 & 616 & 624 &  676 & 864 &  780 &812& 840\\
				 MG: $\Proj$  &   66 & 44 & 41 & 41 & 62 &  42 &35& 42\\
                 MG: $\Proj^*$   & 66 & 39 & 42  & 42 & 62 &  41 &36& 42\\
                 \midrule
				 tot.  MG   & 828 &  699 & 707& 759 & 988 & 863  &883 & 924\\
				\bottomrule
			\end{tabular}
			\caption{Number of iterations and application of MG required by PDP for varying inner accuracy $\Lambda_P=\Lambda_{P^*}=\Lambda_{\rm Cheb}=\Lambda_{\rm ppcg}=\Lambda$ for $\nu=10^{-3}$, nDoF=$836\,841$. In addition to total numbers, preconditioner application required for solving $\tilde Hz=s$ (by ppcg) and applying $\Proj$ and $\Proj^*$ (by cg) are listed.} 
			\label{tabularIterations1}
		\end{table}
	All presented algorithms were implemented in the C++ library Spacy\footnote{https://spacy-dev.github.io/Spacy/}.
	For assembly and multigrid preconditioning, the finite element library Kaskade7 \cite{kaskade,GoetschelSchielaWeiser2020} was used, which is based on the DUNE library \cite{Dune}.
	Since the control mass matrix $M_u$ is relatively small due to boundary control, we use $\tilde M_u = M_u$ and a factorization by the sparse linear solver UMFPACK~\cite{Umfpack}.
	Regarding linear elasticity, the library FunG~\cite{fung}  was applied to compute  derivatives via automatic differentiation. For all experiments we chose a total stopping criterion on the relative error  of size $\Lambda_{\rm PDP}=10^{-8}$ for PDP. From MINRES we demand a relative reduction of the preconditioned residual of $10^{-8}$.

		\begin{figure}[h]
			\centering
		\includegraphics[trim=0cm 0cm 0cm 0cm ,clip,width=0.32\textwidth]{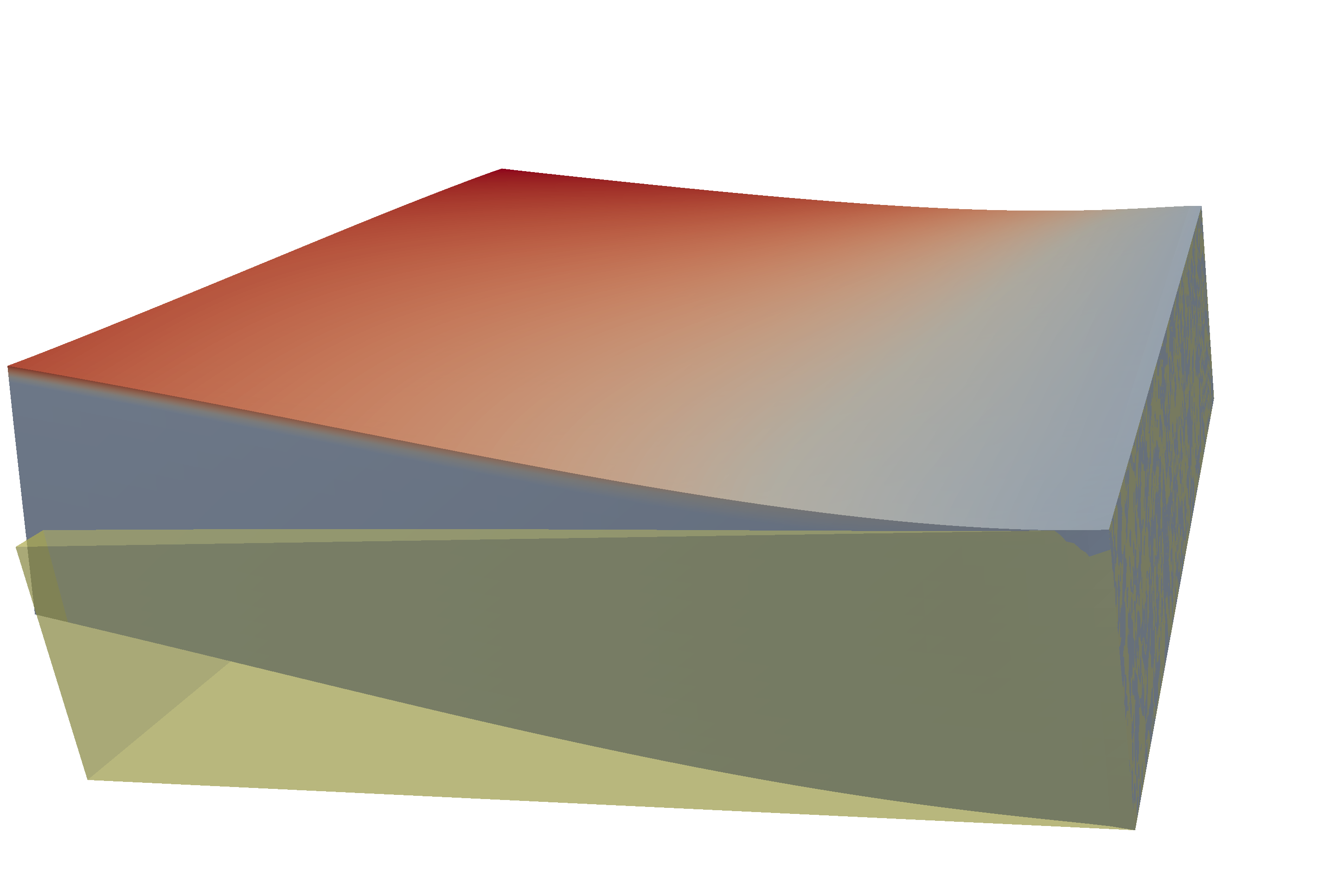}
					\includegraphics[trim=0cm 0cm 0cm 0cm
					,clip,width=0.32\textwidth]{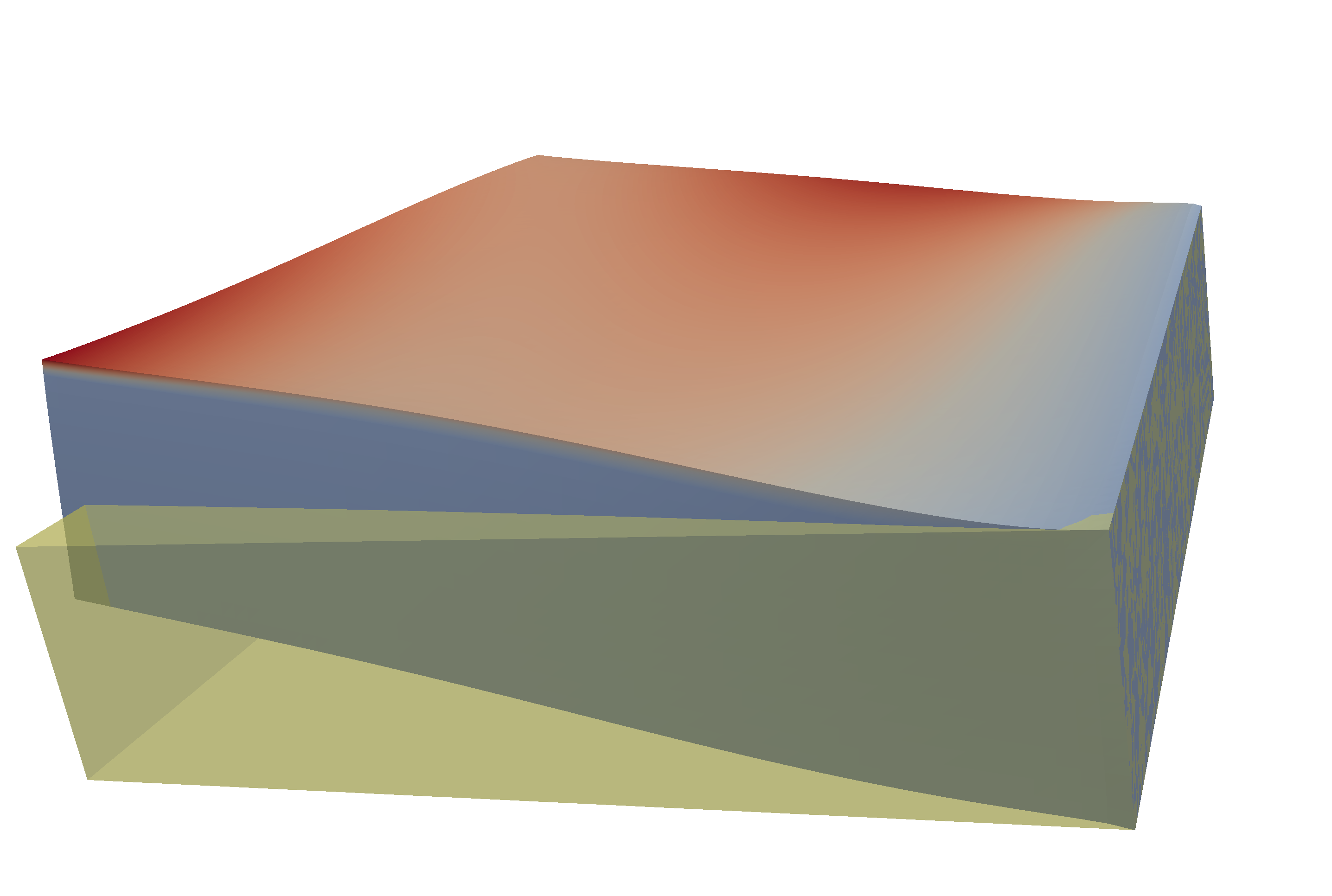} 
					\includegraphics[trim=0cm 0cm 0cm 0cm ,clip,width=0.32\textwidth]{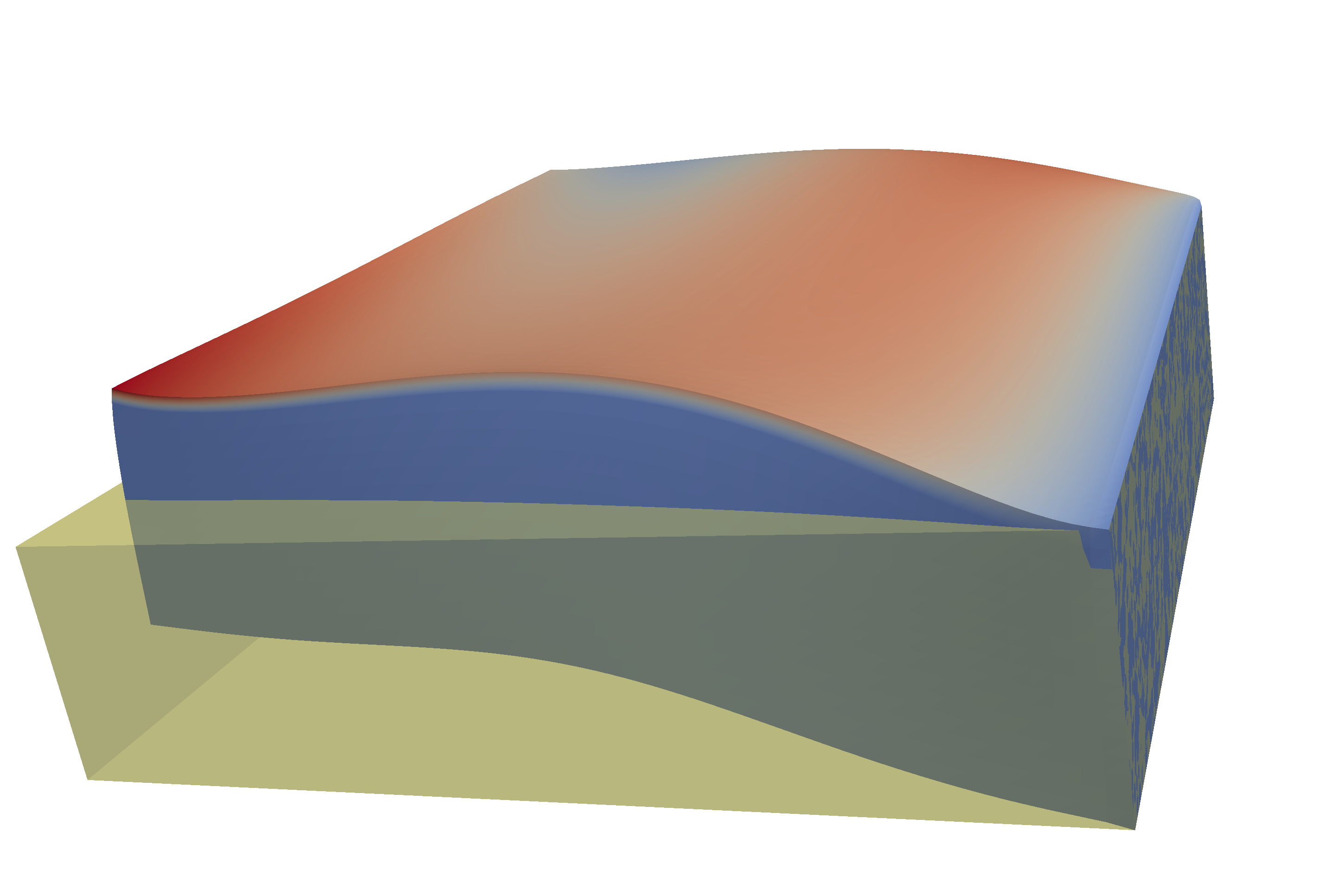} 
					\includegraphics[trim=0cm 0cm 0cm 0cm ,clip,width=0.32\textwidth]{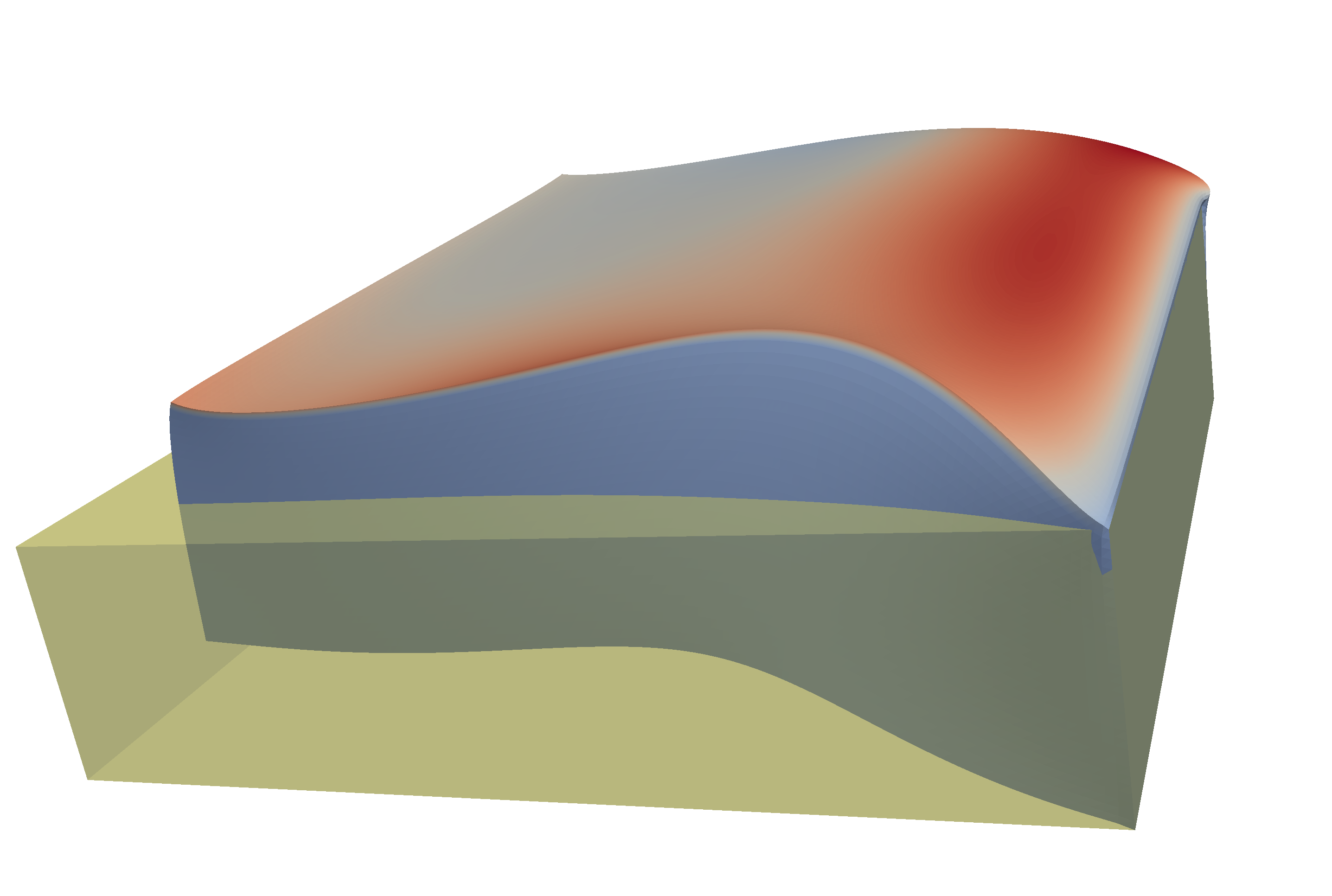} 
					\includegraphics[trim=0cm 0cm 0cm 0cm ,clip,width=0.32\textwidth]{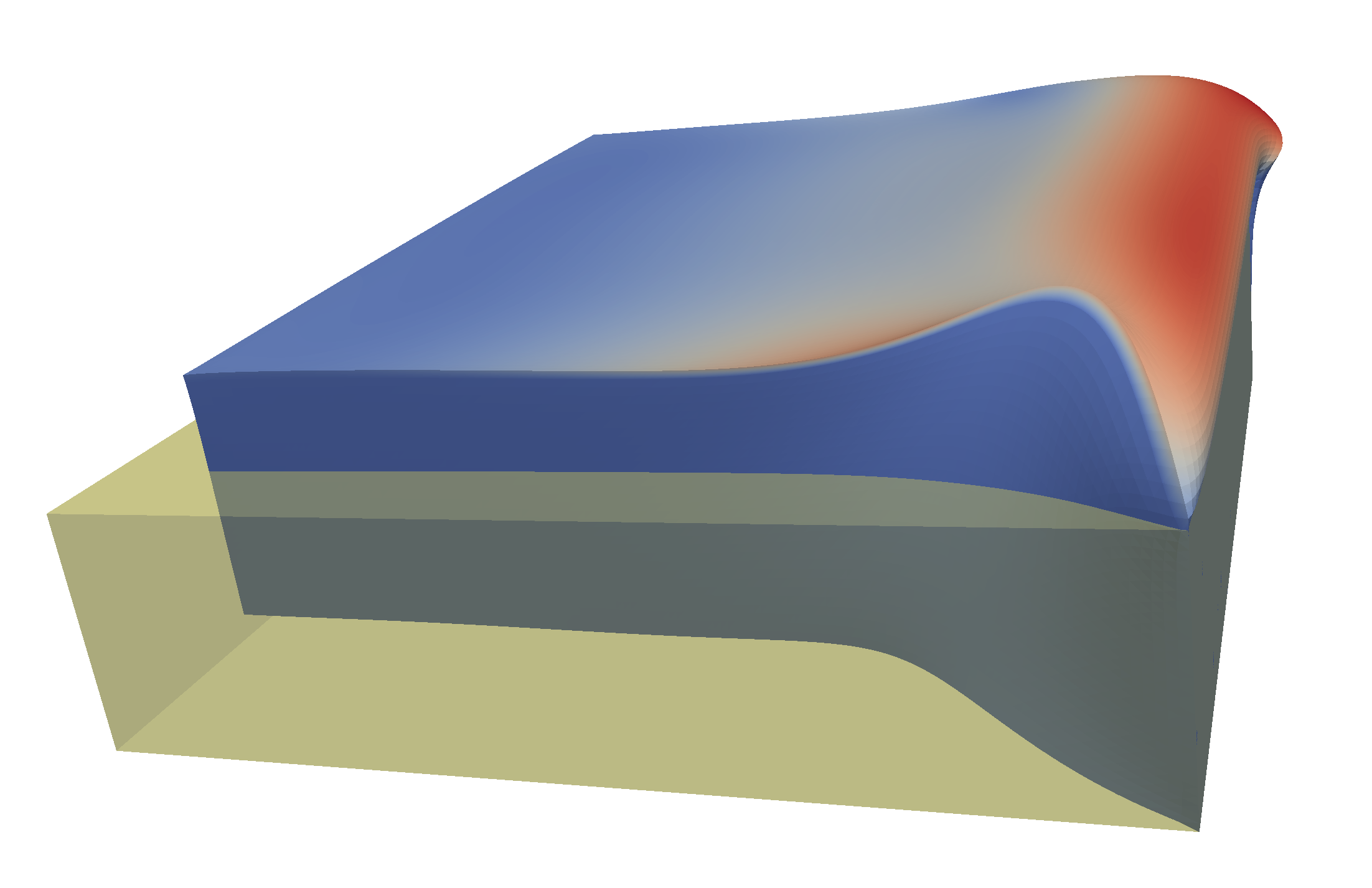} 
					\includegraphics[trim=0cm 0cm 0cm 0cm ,clip,width=0.32\textwidth]{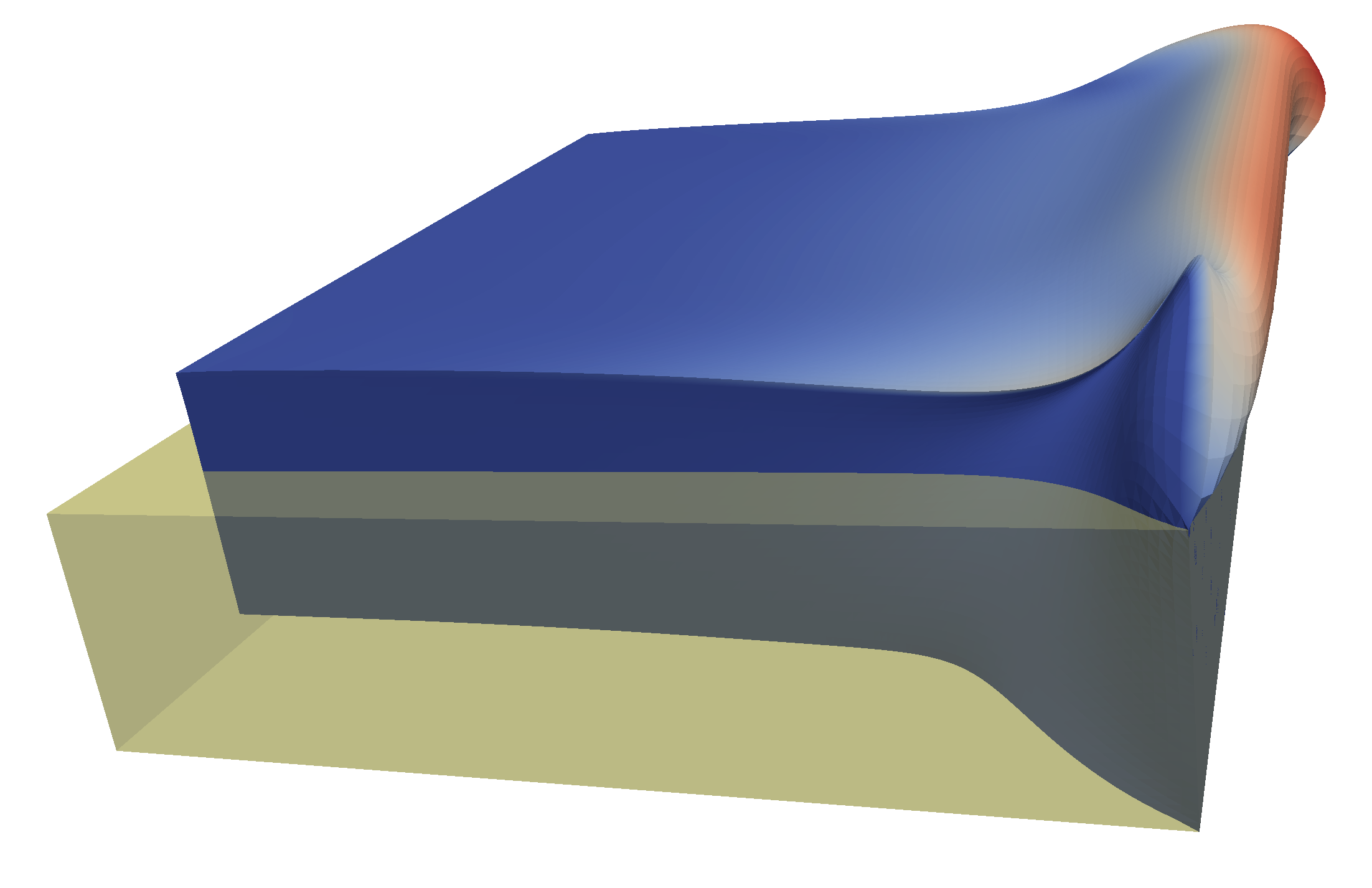} 
					
\caption{Optimal states for $\nu=10^{-k}$, $k=1\dots 6$ (top left to bottom right). Color reflects the relative intensity of the control force (individual scale for each plot). The desired shift (cf. Figure~\ref{figureDesired}) is achieved very well in the left part of the domain, while the Dirichlet boundary conditions on the right face prevent a perfect fit in the right part.   Clearly, the extreme deformations in the last pictures are no longer within the realm of linear elasticity. 
}
			\label{figureIterations}
		\end{figure}

In numerical experiments, profiling information indicates that sparse matrix multiplication with $A$, followed by application of MG or BPX as a preconditioner in Chebyshev and cg dominate the computational effort of PDP, taking around $80\%-90\%$ of the total solution time. The remaining time is spent on the application of $M$, $M_u^{-1}$, $B$, and basic linear algebra. A similar distribution of computational effort has been observed in MINRES. In the spd preconditioner $Q_2$ used for MINRES, application of $M_y^{-1}$ is realized by a diagonally preconditioned Chebyshev iteration (cf.~\cite{wathen2008}) and consumes about $3\%-4\%$ of the total computing time. A similar iteration could be applied to $M_u$ in the case of distributed control.

As can be seen in Table~\ref{tabularIterations3}, BPX requires roughly $3-4$ times more Chebyshev iterations than MG per application of $\tilde A^{-1}$, requiring about half of the computational effort per application. Hence, in total, MG is $1.5$ to $2$ times faster than BPX in our setting. While $Q_{\rm MG}$ shows the expected mesh-independent behaviour, we observe some dependence of the total number of BPX calls in $Q_{\rm BPX}$ on the size of the problem. The reason is that the relative condition number of $A$ preconditioned by BPX is not completely independent of the grid size, but rather seems to grow about an additive constant (around 20) for each refinement, at least for the relatively small number of uniform refinements used.

Next, we study the performance of Algorithm~\ref{alg:PDPOC} for different inner tolerances $\Lambda$. In Table~\ref{tabularIterations1} we compare the number of outer iterations, the number of applications of the MG preconditioner in ppcg (for the Chebyshev solver) and in PDP (for the projections), and the total number of MG calls required for different values of $\Lambda$. As expected, the number of outer iterations becomes smaller for more stringent tolerances. However, the total number of MG calls, while being slightly higher for very stringent inner tolerances, does not vary dramatically over a large range of tolerances. The variation in total number of MG applications is partially due to the discrete effect of varying the low number of outer PDP iterations.

Finally, we consider the performance of PDP with respect to varying Tychonov parameter $\nu$ in the range $10^{-1}$ to $10^{-6}$ (see Figure~\ref{figureIterations} and Table~\ref{tabularIterations2}). We also perform a comparison with MINRES, preconditioned by $Q_1$ and $Q_2$ as defined in~\eqref{eq:minres-preco}, where application of $\tilde A^{-1}$ is, just as in PDP, achieved by MG-preconditioned Chebyshev iterations. We observe throughout that PDP requires a comparable but slightly lower number of preconditioner applications, compared to MINRES. This number grows by a factor of $2-3$, if $\nu$ is reduced by a factor of $10$. For small $\nu$, more stringent inner tolerances seem to be more efficient, in accordance with Proposition~\ref{pro:cond_optctrl}. A moderate tolerance of $\Lambda=10^{-2}$ shows reasonable behaviour over a large range on $\nu$. MINRES$_{Q_1}$ seems to be very sensitive against small values of $\nu$. For $\nu \ge 10^{-5}$ the number of MG applications exceeded by far $100\,000$. However, for moderate $\nu$, MINRES$_{Q_1}$ can be used efficiently with low Chebyshev accuracy, or even with a single MG application for $\tilde A$. In contrast, MINRES$_{Q_2}$ is rather sensitive against inaccurate application of $\tilde A^{-1}$ but more efficient than MINRES$_{Q_1}$ for accurate $\tilde A^{-1}$,  and also more robust with respect to small $\nu$. 

		\begin{table}[h]\small
			\centering
			
			\begin{tabular} { c c c c c c c c }\toprule
			& $\Lambda \backslash \nu$ & $10^{-1}$ & $10^{-2}$ & $10^{-3}$ & $10^{-4}$ & $10^{-5}$ & $10^{-6}$   \\ \midrule
                PDP &$10^{-1}$ & 8(266) & 9(411) & 11(699) & 17(1232) & 26(3704) &  33(8768)   \\
		         $Q_{\rm MG}$ &$10^{-2}$  & 4(320) & 4(398) & 6(759) & 8(1355) & 8(2580) & 10(6427) \\
                 &$10^{-3}$  & 3(381) & 3(445) & 4(863)  & 7(2055)& 7(3473) & 6(8323) \\
				\midrule
                      MINRES &$1$ & 217(434) & 575(1150) & 2851(5702) & 18k(35k) & x & x   \\
                     $Q_{1,\rm MG}$  &$10^{-1}$ & 59(590) & 199(1990) & 1020(10k) & 5981(60k) & x & x   \\
				  & $10^{-2}$  & 52(832) & 184(2944) & 926(15k) & 5262(84k) & x & x \\
               \midrule
                     MINRES &$10^{-1}$ & 875(8750) & 700(7000) & 557(5570) & 500(5000) & 599(5990) & 1208(12k)   \\
				 $Q_{2,\rm MG}$ & $10^{-2}$  & 77(1232) & 74(1184) & 77(1232) & 108(1728) & 209(3344) & 528(8448) \\
                        &     $10^{-3}$  & 27(594) & 32(704)& 47(1034) & 93(2046)  & 176(3872) & 438(9636) \\
                \bottomrule
			\end{tabular}
			\caption{Required iterations for the PDP method for varying Tychonov parameter $\nu$, with varying choices of $\Lambda$ and a grid with 4 refinements. The entries in the table are of the format: outer iterations(total calls $Q^{-1}$). For PDP, $\Lambda$ denotes a common inner tolerance, while for MINRES it denotes the requested Chebyshev accuracy of $\tilde A^{-1}$. $\Lambda =1$ yields one MG V-cycle per application of $\tilde A^{-1}$. }
			\label{tabularIterations2}
		\end{table}
		
\section{Conclusion and Outlook}

In summary, PDP offers competitive efficiency and enjoys a couple of desirable qualitative properties, making it an attractive new option as an iterative solver for large scale equality constrained optimization problems. In particular, it may be used as an inner solver inside SQP methods for the solution of nonlinear equality constrained variational and optimal control problems. The design of a strategy that governs the accuracy requirements of the inner solver efficiently and handles local non-convexity is subject to current work. Also, the application of PDP to time-dependent problems, using a hierarchy of spatial and temporal grids (cf. e.g. \cite{Vexler2011}), is a promising topic of future research. Beyond the scope of linear algebra, it is attractive to combine the current method with an adaptive grid refinement procedure, starting from a coarse grid and refining during the primal projection steps. This would yield an adaptive multilevel method in function space. 

More broadly, the general PDP strategy of solving an inexpensive problem on a surrogate subspace and correcting via primal and dual projections can be extended beyond the realm of PDE constrained optimization and may become a useful tool also in other contexts.

	\bibliographystyle{abbrv}
	\bibliography{mybib}
\end{document}